\documentclass[11pt]{article}
\usepackage{hyphenat}
\usepackage{setspace}
\usepackage[a4paper,margin=1in]{geometry}

\usepackage{amsmath, amssymb, amsthm, mathtools}

\usepackage{bm}

\usepackage{microtype}

\usepackage{hyperref}

\usepackage{xcolor}

\usepackage{mathrsfs}

\usepackage{comment}

\newcommand{\defeq}{\mathrel{\mathop:}=}

\renewenvironment{abstract}{%
	\centering\small
	\textbf\abstractname
	\list{}{\leftmargin0.2cm \rightmargin\leftmargin}
	\item\relax
}{%
	\endlist \par\bigskip
}
\advance\oddsidemargin-\textwidth
\evensidemargin=\oddsidemargin

\newtheorem{theorem}{Theorem}[section]

\newtheorem{scholium}[]{Scholium}

\newtheorem{proposition}[theorem]{Proposition}

\newtheorem{corollary}[theorem]{Corollary}

\theoremstyle{definition}

\newtheorem{definition}[theorem]{Definition}

\newtheorem{example}[]{Example}

\theoremstyle{remark}

\numberwithin{equation}{section}

\newcommand{\intav}[1]{\mathchoice {\mathop{\vrule width 6pt height 3 pt depth  -2.5pt
\kern -8pt \intop}\nolimits_{\kern -6pt#1}} {\mathop{\vrule width
5pt height 3  pt depth -2.6pt \kern -6pt \intop}\nolimits_{#1}}
{\mathop{\vrule width 5pt height 3 pt depth -2.6pt \kern -6pt
\intop}\nolimits_{#1}} {\mathop{\vrule width 5pt height 3 pt depth
-2.6pt \kern -6pt \intop}\nolimits_{#1}}}
\title{%Improved estimates at degenerate nodal sets to fully nonlinear two-phase models with gradient nonlinearities: A scaling approach
A fully nonlinear two-phase Alt-Phillips problem with gradient nonlinearities}
\author{Junior da Silva Bessa, João Vitor da Silva, Yuwei Hu
$\&$
Mayra Soares}
\date{\today}

\begin{document}

\maketitle

\begin{abstract}
\noindent In this manuscript, we analyze a two-phase free boundary problem of Alt-Phillips-type driven by fully nonlinear elliptic equations with lower-order ingredients. More precisely, we consider equations of the type
\[
 F(x,D^{2}u)+\mathscr{H}(x,Du)=\mathscr{F}(x,u^{+},u^{-}) \quad \text{in } B_1 \subset \mathbb{R}^{n},
\]
where the Hamiltonian term exhibits mixed linear and nonlinear gradient dependence, namely
\[
\mathscr{H}(x,\xi):= \langle \mathcal{B}(x),\xi\rangle+\varrho(x)|\xi|^{\sigma}, \quad 0< \sigma\leq2, \,\,\, \sigma \neq 1,
\]
and the reaction is governed by a two-phase power-type nonlinearity of semilinear-type
\[
\mathscr{F}(x,u^{+},u^{-}):=\mathfrak{g}(x)\big[(u^{+})^{m}-(u^{-})^{m}\big] \quad \text{for} \quad  0<m<1.
\]
Under suitable structural assumptions on the coefficients and the operator $F$, we develop a robust analytical framework to capture the fine properties of solutions near the free boundary. In particular, we establish improved growth estimates at higher-order singular nodal points, points where both phases meet and the solution exhibits critical degeneracy. The results are based upon a fine blow-up argument, stability and a Liouville-type theorem. We further prove quantitative non-degeneracy results, ensuring that solutions detach from zero at a controlled rate on each phase. As a consequence of our analysis, we derive a Liouville-type theorem for global solutions within this class. Our results extend and unify several previously known scenarios, and remain new even in the presence of gradient-dependent nonlinearities of order ($\sigma \neq 1$), thereby encompassing models with nonlinear drift and absorption effects arising in phase transition and reaction-diffusion phenomena.

\end{abstract}

\vspace{0.2cm}
	

\section{Introduction}

%Schauder theory made \emph{a priori } estimates a central tool in elliptic PDEs, providing quantitative bounds on solutions without requiring explicit formulas. These estimates play a key role in both linear and nonlinear equations; see, e.g., \cite{Kichenassamy2006}. Over the years, several approaches to Schauder estimates have been developed, including potential theory, Caccioppoli-type energy estimates, maximum principles, Simon's blow-up arguments, mollification techniques, and perturbation methods; see \cite{FR-O} for a detailed overview.
Schauder estimates are a fundamental tool in elliptic regularity, providing quantitative control of higher derivatives in terms of the solution and the regularity of the equation's data. For linear uniformly elliptic equations in non-divergence form, i.e.,
\[\operatorname{Tr}\big(\mathfrak{A}(x)D^2u\big)=f(x)\quad\text{in }B_1, \quad \text{with} \quad \mathfrak{A}=(a_{ij}) \quad \text{such that} \quad \lambda|\xi|^2\leq\sum_{i,j=1}^n a_{ij}(x)\xi_i\xi_j\leq\Lambda|\xi|^2,\]
classical Schauder theory yields \(\mathcal{C}^{2,\alpha}\) estimates under Hölder continuity of the coefficients and right-hand side, 
\[
\|u\|_{\mathcal{C}^{2,\alpha}(B_{1/2})}
\leq
\mathrm{C}\left(
\|u\|_{L^\infty(B_1)}
+
\|f\|_{\mathcal{C}^{0,\alpha}(B_1)}
\right),
\]
where \(\mathrm{C}>0\) depends only on \(n,\alpha,\lambda,\Lambda\), and the \(\mathcal{C}^{0,\alpha}\)-norms of the coefficients. This estimate serves as a prototype for higher regularity theory in more general elliptic equations. Analogous estimates for fully nonlinear equations require additional structural assumptions on the operator, such as convexity, concavity, or suitable quantitative regularity conditions; see, for instance, \cite{BW21, CC95,dosPrazTei2016} and \cite{Gof24}.

Turning now to the semilinear setting, a key motivation arises from the two-phase Alt–Phillips problem, see \cite{Fot-Shah2017}. More precisely, let $\Omega \subset \mathbb{R}^n$ be a bounded domain with Lipschitz boundary, and consider, for instance, the equation 
\begin{equation}\label{TP-Alt-Phillips}
\Delta u = \gamma\lambda_+u^{\gamma-1}\chi_{\{u>0\}} - \gamma\lambda_-(-u)^{\gamma-1}\chi_{\{u<0\}} \quad \mbox{in } \Omega, \qquad 1<\gamma<2 \quad \text{and} \quad \lambda_{\pm}>0.\tag{A-P}
\end{equation}
Its solutions are obtained as local minimizers of a suitable Alt–Phillips functional. In this context, minimizers may change sign, giving rise to free boundaries between positive, negative, and zero phases. We stress that the one-phase case, namely $u^{-}=0$, corresponds to the seminal work of Alt and Phillips \cite{Alt-Phillips}, where a free boundary problem arising from the study of gas distribution in a porous catalyst pellet was introduced. Although such variational models allow for the use of classical potential-theoretic techniques, sign-changing solutions to \eqref{TP-Alt-Phillips} pose substantial technical difficulties, since standard tools available in the one-phase setting, such as Harnack's inequality, Hopf and boundary point lemmas, and the Strong Maximum Principle, generally fail.

Now, considering associated with \eqref{TP-Alt-Phillips} the scaling
\[
u_r(x):=\frac{u(x_0+rx)}{r^\beta}, 
\]
we may {determine} the natural homogeneity of the equation and observe that the scaling is consistent with the assumption
\(D^2u(x_0)=0\). In fact,
$$\Delta u_r(x) = r^{2-\beta}\Delta u(x_0+rx), \quad \text{while} \quad \vert{}u(x_0+rx)\vert{}^{\gamma-1} = r^{\beta(\gamma-1)}\vert{}u_r(x)\vert{}^{\gamma-1}.$$
The scale invariance requires
{$2-\beta+\beta(\gamma-1)=0,$}
and hence
$
\beta=\frac{2}{2-\gamma}.
$
Since \(1<\gamma<2\), we have
$\beta>2.
$ Thus, if \(x_0\) is a \emph{higher-order degenerate and vanishing point} of u, namely,
$
u(x_0)=
|D u(x_0)| =
|D^2u(x_0)|=0,
$
the natural growth predicted by the scaling of the equation is
$$\vert{}u(x)\vert{} \approx \vert{}x-x_0\vert{}^{\frac{2}{2-\gamma}}, \quad \vert{}D u(x)\vert{} \approx \vert{}x-x_0\vert{}^{\frac{\gamma}{2-\gamma}}, \quad \text{and} \quad \vert{}D^2u(x)\vert{} \approx \vert{}x-x_0\vert{}^{\frac{2(\gamma-1)}{2-\gamma}}.$$
Since \(\gamma>1\), it follows that
$\frac{2(\gamma-1)}{2-\gamma}>0,$
and consequently,
$
D^2u(x)\to 0$, as $x\to x_0.
$

Therefore, the scaling itself doesn't establish \(D^2u(x_0)=0\). Rather, once we are at a point where \(u=Du=D^2u=0\), the scaling identifies the natural next-order homogeneity. This suggests the higher-order expansion, namely,
$
u(x)
=
\text{o}(|x-x_0|^2)
\ \text{as }x\to x_0.
$

Comparing these results with the corresponding estimates from Schauder theory, the right-hand side can be written schematically as
$
f(x)\sim |u(x)|^{\gamma-1}.
$
Under the natural growth
\[
|u(x)|\leq |x-x_0|^\beta,
\quad \mbox{
we obtain } 
\quad
|f(x)| \sim
|x-x_0|^{\frac{2(\gamma-1)}{2-\gamma}}.
\]
Thus, the natural H\"{o}lder exponent of the source term at $x_0$ is given by
$
\alpha_\gamma
=
\frac{2(\gamma-1)}{2-\gamma}.
$
Consequently, whenever $0 < \alpha_\gamma < 1$, classical Schauder theory suggests that $u \in \mathcal{C}^{2,\alpha_\gamma}$ near $x_0$, with the pointwise estimate 
\[
|D^2u(x)|
\le
|x-x_0|^{\alpha_\gamma},
\qquad
\alpha_\gamma
=
\frac{2(\gamma-1)}{2-\gamma},
\]
since \(D^2u(x_0)=0\).
Notice that
$ \beta-2
=
\frac{2}{2-\gamma}-2
=
\frac{2(\gamma-1)}{2-\gamma}
=
\alpha_\gamma,
$
which yields
$\beta=2+\alpha_\gamma.
$
This identity provides a direct structural interpretation of the scaling exponent 
\[
|u(x)|\le |x-x_0|^{2+\alpha_\gamma}
\quad\Longleftrightarrow\quad
|D^2u(x)|\le |x-x_0|^{\alpha_\gamma}.
\]

Thus, at higher-order degenerate and vanishing points, the exponent $\beta = \frac{2}{2-\gamma}$ represents the \emph{intrinsic scaling exponent of the equation}.   This scaling is the free-boundary analogue of the compatibility between the vanishing order of the source term and the \(\mathcal{C}^{2,\alpha}\) scale in Schauder theory. Observe also that, $\beta-2 = \frac{2(\gamma-1)}{2-\gamma},$ corresponds precisely to the Hessian's Hölder exponent predicted by Schauder theory.

In the same spirit, the fully nonlinear analogue of the one-phase Alt–Phillips problem,
$$
F(D^2u) = u^{\gamma-1} \quad \text{in} \quad B_1, \,\,\,\text{for } \gamma \in (1,2) \quad \text{ and } \quad u \ge 0, 
$$ was analyzed by Hu and Wu \cite{Wu-Yu2022}. Under the assumption that the convex operator $F: \operatorname{Sym}(n) \to \mathbb{R}$ satisfies $F(\mathbf{O}_n) = 0$ with its subdifferential at $\mathbf{O}_n$ given by a trace operator, they proved optimal regularity to solutions, namely, $u \in \mathcal{C}_{\text{loc}}^{\frac{2}{2-\gamma}}$, as well as $\mathcal{C}^1$ regularity for the regular part of the free boundary. Motivated by their work, we introduce our problem below.

\subsection{A two-phase fully nonlinear Alt-Phillips model}

Inspired by the previous heuristics, we study a \emph{two-phase Alt-Phillips free boundary problem} governed by fully nonlinear elliptic PDEs with Hamiltonian terms and sublinear (or concave) reaction terms. These models serve as non-variational counterparts to semilinear variational free boundary problems across diverse physical and mathematical contexts (cf. \cite{Fot-Shah2017}).
Precisely, let us consider the following PDE model
\begin{equation}\label{1.1}
 F(x,D^{2}u)+\mathscr{H}(x,Du)=\mathscr{F}(x,u^{+},u^{-}),
\end{equation}
where
\begin{equation}\label{1.2}
\mathscr{H}(x,{\xi}):= \langle \mathcal{B}(x),{\xi}\rangle+\varrho(x)|{\xi}|^{\sigma}, \quad \mbox{for all } x\in B_1, \ \xi \in \mathbb{R}^n
\end{equation}
with $\sigma \in (0, 2] \setminus \{1\}$ and
\begin{equation}\label{Assump-RHS}
\mathscr{F}(x,u^{+},u^{-}):=\mathfrak{g}(x)((u^{+})^{m}-(u^{-})^{m}), \quad \text{for} \quad  0<m<1.
\end{equation}

Additionally, we will assume the following structural assumptions:

\begin{itemize}

\item[{\bf ($\mathrm{A}_1$)}] ({\bf Uniform ellipticity})
Let
$F:\mathrm{Sym}(n)\times  B_1 \to \mathbb{R}
$
be a continuous, fully nonlinear operator. We assume that there exist constants
\(0<\lambda\leq \Lambda\) such that
\begin{equation*}\label{UnifEll}
\mathcal{M}^{-}_{\lambda,\Lambda}(\mathrm{Y})
\leq
F(\mathrm{X}+\mathrm{Y},x)-F(\mathrm{X},x)
\leq
\mathcal{M}^{+}_{\lambda,\Lambda}(\mathrm{Y}),
\end{equation*}
for every \(x\in B_1\) and every
\({\rm X,Y}\in \mathrm{Sym}(n)\) with \(\rm Y\geq 0\).
Here, \(\mathcal{M}^{\pm}_{\lambda,\Lambda}\) denote the Pucci extremal operators
\[
\mathcal{M}^{+}_{\lambda,\Lambda}(\rm X)
=
\Lambda \sum_{e_i>0} e_i
+
\lambda \sum_{e_i<0} e_i,
\mbox{
and \ }
\mathcal{M}^{-}_{\lambda,\Lambda}(X)
=
\lambda \sum_{e_i>0} e_i
+
\Lambda \sum_{e_i<0} e_i,
\]
where \(e_i = e_i(\mathrm{X})\) are the eigenvalues of \(\mathrm{X}\). In addition,
$F({ \bf O_n},x)=0
\ \text{for every } x\in B_1.
$

\item[{\bf ($\mathrm{A}_2$)}] ({\bf Regularity of the lower-order coefficients})
The coefficients satisfy
\[
\mathcal{B}\in \mathcal{C}^{0,\alpha_0}( B_1;\mathbb{R}^{n})
\cap L^{\infty}( B_1;\mathbb{R}^{n}),
\varrho\in \mathcal{C}^{0,\alpha_0}( B_1)\cap L^{\infty}( B_1),
\mbox{
and
}
\mathfrak g\in \mathcal{C}^{0}( B_1)\cap L^{\infty}( B_1)
\] for some exponent $\alpha_0 \in (0, 1)$. Setting $\mathfrak{F}: = \mathfrak{g}^{-1}(0)$, it is a close subset of $B_1$, and 
there exist  constants $\mathrm{c}_0>0,$ and 
$ \mathrm{C}_0>0$ and $\mu\in \mathbb{R}$, satisfying $0\leq\mu\leq 1 -3m $,  such that
\[ 
\mathrm{c_0}\operatorname{dist}(x,\mathfrak{F})^\mu 
\leq \mathfrak{g}(x) \leq \mathrm{C_0}\operatorname{dist}(x,\mathfrak{F})^\mu \ \mbox{ in }  B_1.
\]

\item[{\bf ($\mathrm{A}_3$)}] ({\bf Continuity of the operator})
There exists a modulus of continuity \(\tau\) such that
\begin{equation*}\label{elliptic-continuity}
\frac{|F(\rm X,x)-F(\rm X,y)|}
{\|\mathrm{X}\|}
\leq
\tau(|x-y|),
\end{equation*}
for every \(\rm X\in \mathrm{Sym}(n)\), and every
\(x,y\in B_1\).
Equivalently, defining
\[
\Phi_{\mathrm{F}}(x,y)
:=\sup_{\mathrm{X} \in \text{Sym}(n) \atop{\mathrm{X} \neq \mathbf{O}_n}} \frac{|F(\mathrm{X}, x)-F(\mathrm{X}, 0)|}{\|\mathrm{X}\|},
\]
we have $\Phi_{\mathrm{F}}(x,y)
\leq
\tau(|x-y|).
$

\item[{\bf ($\mathrm{A}_4$)}] ({\bf Differentiability of the operator})
For every \(x\in B_1\), the map
$
F(\cdot,x):\mathrm{Sym}(n)\to\mathbb{R}
$ is of class \(\mathcal{C}^{1}\), and there exists a modulus of continuity \( \omega: [0, +\infty) \to [0, +\infty)\) such that
\[
|D_{\rm X}F({\rm Z},x)-D_{\rm X}F({\rm Y},x)|
\leq
 \omega(|\rm Z-Y|),
\]
for every \(x\in B_1\) and every
\(\rm Z,Y\in\mathrm{Sym}(n)\).
Here,
\[
D_{\rm X}F({\rm X_{0}},x)
:=
\lim_{s\to 0}
\frac{
F({\rm X_{0}}+s{\rm X},x)-F({\rm X_{0}},x)
}{s} = \text{Tr}(\mathfrak{A}(x){\rm X_{0}})
\]
denotes the G$\hat{a}$teaux derivative of \(F\) with respect to the matrix variable.

\end{itemize}

Despite recent advances in variational and non-variational settings, optimal regularity remains open for two-phase problems driven by quasilinear operators in non-divergence form with drift terms and doubly homogeneous reaction terms as $\mathscr{F}(u^+,u^-,x)$, considered here. The non-variational, nonlinear character of this problem poses significant analytical obstacles.

The purpose of this work is to overcome such obstacles to show that the scaling prediction can indeed be recovered as an a priori estimate. More precisely, under suitable structural assumptions on the fully nonlinear operator and the lower-order terms, we prove an improved growth estimate at points where the solution and its first two derivatives vanish. The proof combines blow-up techniques, compactness, higher-order regularity estimates, and a Liouville-type classification for the limiting equation.

{In summary, our main contributions can be outlined as follows:}

\begin{itemize}
\item[(i)] {\textbf{Improved growth estimate at higher-order singular nodal points} -- Theorem~\ref{Thm:branching-elliptic}. This result is new in the nondivergence setting and remains novel even in the linear case. It is also closely related to classical membrane problems, as well as to dead-core phenomena;}

\item[(ii)] {\textbf{Sharp gradient growth } -- Corollary~\ref{Cor:SharpGradientGrowth}. This result relies crucially on the gradient estimates established previously by the authors; see~\cite[Theorem 1.6]{daSN2021};}

\item[(iii)]{\textbf{$L^{p}$-average estimates of the Hessian} -- Theorem~\ref{Lpaverage}. An improved estimate for the second derivatives near higher-order singular nodal points in the $L^p$-average sense, building crucially on the second author's previous work \cite[Theorem~6.7]{daSN2021};}

\item[(iv)] {{\bf Boundary estimate at higher-order singular nodal points} -- Theorem~\ref{Thm:boundary}. This result provides the boundary counterpart of Theorem~\ref{Thm:branching-elliptic}, showing that a prescribed fine estimate on the boundary propagates into the interior for this class of two-phase models. The proof combines a blow-up argument, compactness, and a Liouville theorem for solutions in a half-space;}

\item[(v)] {\textbf{Flipping estimates } -- Theorem~\ref{ThmFlipElliptic}. Consists of obtaining estimates for one component (either the positive or negative
part) of a solution, provided suitable bounds are available for the complementary
component;}

\item[(vi)] {\textbf{Non-degeneracy of solutions} -- Theorem~\ref{Non-degeneracy of solutions}. The proof relies on a comparison principle established previously by the authors; see~\cite[Theorem 1.7]{BBdaSSS25}.}

\item[(vii)] {\textbf{Further consequences}: We also provide applications that illustrate the scope of our results, including a Liouville-type theorem -- Theorem~\ref{thm:liouville}.}

\end{itemize}

\subsection{Our main results and consequences}

In this subsection, we present our main contributions. 
Recall that the nodal set of $u$ is defined by $\mathcal{N}(u) := \{u = 0\}$. In classical nodal theory, it decomposes into the regular and singular nodal sets $$\mathcal{R}(u) := \mathcal{N}(u) \cap \{\vert{}Du\vert{} \neq 0\} \quad \text{and} \quad \mathcal{S}(u) := \mathcal{N}(u) \cap \{\vert{}Du\vert{} = 0\}.$$ Within this framework, our set of interest,$$\Gamma_2(u) := \{\vert{}Du\vert{} = \vert{}D^2u\vert{} = 0\} \cap \partial\{u > 0\} \cap \partial\{u < 0\},$$ is a subset of $\mathcal{S}(u)$. More precisely, it corresponds to the \emph{higher-order singular nodal set}, as both first- and second-order derivatives vanish at its points, a terminology consistent with classical nodal theory and recent a priori estimates for degenerate equations.

In addition, we consider the class of solutions to \eqref{1.1}:
\begin{equation*}
\mathcal{G}_{2}^{\alpha_0} := \left\{
u \in \mathcal{C}^{0}(B_1) \; \mbox{solving   (\ref{1.1}) \ } : \ \|u\|_{\mathcal{C}^{2,\alpha_0}(B_{1/2})}\leq \mathrm{K}_0(\mathrm{M}_0,\text{data})
, \ \alpha_0 \in (0,1)\right\},
\end{equation*}
where the constant $\mathrm{M}_{0}>0$ satisfies $\max\{\|u\|_{L^{\infty}(B_{1})},\|\mathcal{B}\|_{\mathcal{C}^{0,\alpha_{0}}(B_{1};\mathbb{R}^{n})},\|\varrho\|_{\mathcal{C}^{0,\alpha_{0}}(B_{1})}\}\leq \mathrm{M}_{0}$. Under structural conditions ($\mathrm{A}_1$)–($\mathrm{A}_4$), $\mathcal{G}_{2}^{\alpha_0} \neq \emptyset$ by \cite{BdaDO2026, daSN2021}.
We are now ready to state our first main result.

\begin{theorem}[\bf Improved growth estimate at higher-order singular nodal points]
\label{Thm:branching-elliptic}
Let $u\in \mathcal{G}_{2}^{\alpha_0}$ be a solution of \eqref{1.1} in $B_1$.  % with $\mathfrak{g}(x):=|x|^\mu$. 
Fix $ \beta\in \left(2+\alpha_0,\frac{2+\mu}{1-m}\right]$ and $\sigma\in\left[\tfrac{\beta-2}{\beta-1},1\right)\cup(1,2]$. Assume that the structural hypotheses $(\mathrm{A}_1)-(\mathrm{A}_4)$ are satisfied. Given $x_0\in\Gamma_2(u)\cap \mathfrak{g}^{-1}(0)\cap B_{1/2}$, there exists a constant $\mathrm{C}>0$ depending only on
$n$, $\lambda$, $\Lambda$, $m$, $\sigma$,
$\|\mathcal{B}\|_{\mathcal{C}^{0,\alpha_{0}}(B_1;\mathbb{R}^{n})}$,
$\|\varrho\|_{\mathcal{C}^{0,\alpha_{0}}(B_1)}$, and 
$\|\mathfrak g\|_{L^\infty(B_1)}$,
 such that
\[
|u(x)|
\le \mathrm{C}\|u\|_{L^\infty(B_1)}|x-x_{0}|^{\beta},
\quad  \forall \ x\in B_{r}(x_{0}),\ \forall \ r\in (0,1/2).
\]
Consequently,
\[
\sup_{B_{r}(x_{0})}|u|\leq \mathrm{C}\|u\|_{L^{\infty}(B_{1})}r^{\beta},\quad \forall \ r\in (0,1/2).
\]
\end{theorem}

For convenience, we set $\bar\alpha \defeq \frac{2+\mu}{1-m}$, and we shall use this notation throughout. The estimate in Theorem
\ref{Thm:branching-elliptic} shows that $\bar\alpha$ is the natural
vanishing order dictated by the scaling of the equation. As a consequence,
the corresponding gradient decay at higher-order singular nodal points is
given by the following corollary.

\begin{corollary}[\bf Sharp Gradient Growth]
\label{Cor:SharpGradientGrowth}
Under the assumptions of Theorem~\ref{Thm:branching-elliptic}, if additionally, $\sigma \in \left(1 +  \tfrac{1-m}{1+\mu +m},2\right]$
there exists a universal constant $\mathrm{C}>0$ such that
\[
\sup_{B_r(x_0)} |Du|
\le
\mathrm{C}\, r^{\bar\alpha-1},
\]
for every $x_0\in\Gamma_2(u)\cap B_{1/2}$ and all
$0<r<\frac14$. Equivalently,
\[
\sup_{B_r(x_0)} |Du|
\le
\mathrm{C}\, r^{\frac{1+\mu+m}{1-m}}.
\]
\end{corollary}

The approach developed in this work is sufficiently flexible to yield several further applications concerning regularity at points of the higher-order singular nodal set $\Gamma_2(u)$. Among these,  
we establish a novel \(L^p\)-average estimate of the Hessian for viscosity solutions of \eqref{1.1} for any $n<p<\infty$, when $F$ is a convex or concave operator. 

\begin{theorem}[{\bf $L^{p}$-average estimates of the Hessian}]\label{Lpaverage}
Let $u\in \mathcal{G}_{2}^{\alpha_0}$ be a solution of \eqref{1.1} in $B_{1}$, $n<p<\infty$, fix $\beta\in (2+\alpha_{0},\bar{\alpha}]$ and $\sigma\in \left[\frac{\beta-2}{\beta-1},1\right)\cup (1,2]$. Assume that the structural conditions $\rm (A_{1})-(\rm A_{3})$ are valid and that $F$ is a convex or concave operator.  Given $x_{0}\in \Gamma_2(u)\cap \mathfrak{g}^{-1}(0)\cap B_{1/2}$, there exists a universal constant $\mathrm{M}_{0}>0$, depending on $n$, $p$, $\lambda$, $\Lambda$, $m$, $\sigma$, $\|\mathcal{B}\|_{\mathcal{C}^{0,\alpha_{0}}(B_{1};\mathbb{R}^{n})}$, $\|\varrho\|_{\mathcal{C}^{0,\alpha_{0}}(B_{1})}$ and $\|u\|_{L^{\infty}(B_{1})}$ such that
\[
\left(\intav{B_{\frac{r}{2}}(x_{0})}|D^{2}u|^{p}\,dx\right)^{\frac{1}{p}}\leq \mathrm{M}_{0}r^{\beta-2},\quad \text{for every } r\in (0,1/2).
\]
\end{theorem}

As another consequence of Theorem \ref{Thm:branching-elliptic}, we obtain a Liouville-type result for solutions exhibiting controlled growth at infinity.

\begin{theorem}[{\bf Liouville-type result}]\label{thm:liouville}
Let \( u \in \mathcal{C}^0(\mathbb{R}^n) \) be a viscosity solution of
\[
F(x,D^{2}u) + \mathscr{H}(D u, x) = \mathscr{F}(x,u^{+},u^{-}) \quad \text{in  }\ \mathbb{R}^n,
\]
under assumptions \((\mathrm{A}_1)-(\mathrm{A}_4)\). Suppose \( x_0\in\Gamma_2(u)\cap \mathfrak{g}^{-1}(0)\cap B_{1/2}\) and
\begin{equation}\label{Hip-Liouville:Thm}
\lim_{|x| \to +\infty} \frac{|u(x)|}{|x - x_0|^{\bar{\alpha}}} = 0.    
\end{equation}
Then, \( u \equiv 0 \).
\end{theorem}

In addition, under appropriate conditions, we establish that regularity at higher-order singular nodal points holds up to the boundary. Precisely, we aim to investigate the following model:
\begin{equation}\label{boundeq}
\left\{
\begin{array}{rcll}
F(x,D^{2}u) + \mathscr{H}(x, D u) & = & \mathfrak{g}(x)\big[ (u^+)^m -(u^-)^m \big] & \text{in }\ B_{1}^{+},\\
u & = & \mathrm{h} & \text{on }\ T_{1},
\end{array}
\right.
\end{equation}
where, for $x_{0}\in T_{1}$ and $r>0$, we denote
\begin{equation*}
B_{r}^{+}(x_{0}) := B_{r}(x_{0})\cap \mathbb{R}^{n}_{+}
\qquad\text{and}\qquad
T_{r}(x_{0}) := B_{r}(x_{0})\cap \{x_{n}=0\}.
\end{equation*}
In particular, $B_{1}^{+}=B_{1}^{+}(0)$ and $T_{1}=T_{1}(0)$.

To state the boundary regularity result, we introduce the following class of functions associated with problem \eqref{boundeq}:
\begin{equation*}
\widetilde{\mathcal{G}}_{2}^{\alpha_0}
:=
\left\{
u \in \mathcal{C}^{0}(B^{+}_1\cup T_{1}) \;\mbox{solving \eqref{boundeq} } : \
\|u\|_{\mathcal{C}^{2,\alpha_0}(\overline{B^{+}_{1/2}})}
\leq \mathrm{K}_0(\widetilde{\mathrm{M}}_0,\text{data}),
\ \alpha_0\in(0,1)
\right\},
\end{equation*}
where $\widetilde{\mathrm{M}}_{0}>0$ is such that $\max\left\{
\|u\|_{L^{\infty}(B^{+}_{1})},
\|\mathcal{B}\|_{\mathcal{C}^{0,\alpha_{0}}(B^{+}_{1};\mathbb{R}^{n})},
\|\varrho\|_{\mathcal{C}^{0,\alpha_{0}}(B^{+}_{1})}
\right\}
\leq \widetilde{\mathrm{M}}_{0}$. Assume further that $\mathrm{h} \in \mathcal{C}^{2,\widehat\beta-2}(T_1)$ for some $\widehat\beta>2+\alpha_0$ and that
\[
\mathrm{h}=|D\mathrm{h}|=|D^2\mathrm{h}|=0
\]
at higher-order singular nodal points. Under these assumptions, the higher-order vanishing of the boundary data propagates into the interior, yielding the following boundary regularity estimate.

\begin{theorem}[{\bf Boundary estimate at higher-order singular nodal points}]\label{Thm:boundary}
Suppose that the hypotheses of Theorem~\ref{Thm:branching-elliptic} are satisfied, and consider $\mathrm{h} \in \mathcal{C}^{2,\widehat\beta-2}(T_1)$, for some $\widehat\beta > 2 + \alpha_0$.  Let $u\in \widetilde{\mathcal{G}}_{2}^{\alpha_{0}}$ be a viscosity solution of \eqref{boundeq}. Given
$x_0\in\Gamma_2(u)\cap \mathfrak{g}^{-1}(0) \cap T _{1/2}$, there exists a constant $\mathrm{C}>0$ depending only on 
$n$, $\lambda$, $\Lambda$, $m$, $\sigma$,
$\|\mathcal{B}\|_{\mathcal{C}^{0,\alpha_{0}}(B_1^+;\mathbb{R}^{n})}$,
$\|\varrho\|_{\mathcal{C}^{0,\alpha_{0}}(B_1^+)}$,
$\|\mathfrak g\|_{L^\infty(B_1^+)}$, and
$\|\mathrm{h}\|_{\mathcal{C}^{2,\widehat\beta -2}(T _1)}$,
such that
\[
|u(x)|
\le \mathrm{C}\|u\|_{L^{\infty}(B_{1}^{+})}|x-x_{0}|^{\beta},
\quad \forall \ x\in B_{r}^{+}(x_{0})\cup T_{r}(x_{0}), \ \forall \ r\in(0,1/2).
\]
Consequently,
\[
\sup_{B_{r}^{+}(x_{0})\cup T_{r}(x_{0})}|u|\leq \mathrm{C}\|u\|_{L^{\infty}(B_{1})}r^{\beta},\quad \forall \ r\in (0,1/2).
\]
\end{theorem}

{The preceding result is new in the nondivergence two-phase setting with Hamiltonian
terms and remains novel even in the linear case. Moreover, it is closely related to
the classical global Schauder estimates established by the authors in their recent
work \cite{BdaDO2025}}.   

\medskip

Inspired by \cite{BF24}, next result shows that, at points $x_{0}\in \Gamma_2(u)\cap \mathfrak{g}^{-1}(\{0\})$, if the positive (resp. negative) part of $u$ decays in $B_{r}(x_{0})$ at the rate $r^{\bar{\alpha}}$, then the same decay holds for the other part. This property is known as \textit{flipping estimates} and is summarized in the following result.

\begin{theorem}[\bf Flipping estimates]\label{ThmFlipElliptic}
Let $u\in \mathcal{C}^{0}(B_1)$ be a viscosity solution to
\eqref{1.1} and  $x_{0}\in \Gamma_2(u)\cap \mathfrak{g}^{-1}(\{0\})$. Assume  the structural conditions $(\mathrm{A}_1)-(\mathrm{A}_4)$ hold, and
\(
\ \sigma\in\left(\frac{2m+\mu}{1+\mu+m},1\right).
\) Suppose also that there exists $r_{0}>0$ such that
\(
B_{r_{0}}(x_{0})\Subset B_1.
\)
If
\[
\sup_{B_r(x_{0})} u^{+}
\leq
\mathrm{C}_{0}r^{\bar{\alpha}}
\qquad \text{for all } r\in(0,r_{0}],
\]
then there exists a constant $\mathrm{C}_{1}>0$ such that
\[
\sup_{B_r(x_{0})} u^{-}
\leq
\mathrm{C}_{1}r^{\bar{\alpha}}
\qquad \text{for all } r\in(0,r_{0}].
\]
The analogous conclusion remains valid after interchanging the roles of
$u^{+}$ and $u^{-}$.
\end{theorem}

Finally, regarding the analysis at higher-order singular nodal points, we also obtain the non-degeneracy of the solutions in each phase along these points.

\begin{theorem}[{\bf Non-degeneracy of solutions}]\label{Non-degeneracy of solutions}
Let \(u \in \mathcal{C}^{0}(B_1)\) be a viscosity solution of \eqref{1.1}. Suppose that  {$(\mathrm{A}_1)-(\mathrm{A}_2)$} hold.
Then, for every \(x_{0}\in \Gamma_2(u)\cap \mathfrak{g}^{-1}(\{0\})\),
\[
\sup_{B_r(x_0) \cap \{u>0\}} u(x) \geq \mathrm{C}_1\, r^{\bar{\alpha}}, 
\quad \text{and} \quad 
\inf_{B_r(x_0) \cap \{u<0\}} u(x) \leq -\mathrm{C}_2\, r^{\bar{\alpha}},
\]
for all \( r < \mathrm{dist}(x_0, \partial B_1) \), where the constants \(\mathrm{C}_1,\mathrm{C}_{2}> 0\) {depend} only on  $n$, $\Lambda$, $\mathrm{C}_{0}$, $\|\mathcal{B}\|_{L^{\infty}(B_1;\mathbb{R}^{n})}$, $\|\varrho\|_{L^{\infty}(B_1)}$, $\sigma$, $m$,  and $\mu$.
\end{theorem}

\subsection{Some examples and further comments}

Concerning our main results, we first emphasize that the structural conditions $(\mathrm{A}_1)-(\mathrm{A}_4)$ are primarily designed to guarantee Schauder-type estimates for solutions of \eqref{1.1} and stability under the blow-up procedure: suitable limits of operators in this class remain within a class for which the corresponding Liouville property holds. Throughout this work, we say that an operator $F$ enjoys the \textbf{Liouville property} if every viscosity solution $u$ of
\[\begin{cases}
F(D^{2}u)=0 \quad \text{in } \mathbb{R}^{n},\\
\displaystyle\lim_{|x|\to\infty}\tfrac{u(x)}{|x|^{2+\alpha_{0}}}=0,
\end{cases}
\]
is a quadratic polynomial. Consequently, our arguments  are not restricted to the particular class characterized by $(\mathrm{A}_1)-(\mathrm{A}_4)$, they also apply to any other class sharing the same stability and Liouville properties. For instance, this includes:
\begin{itemize}
\item[1.] Convex or concave operators; see, for example, Caffarelli-Cabr\'{e}'s book \cite{CC95};
\item[2.] Quasiconvex or quasiconcave operators, see \cite{Gof24};
\item[3.] Uniformly elliptic operators in the plane $\mathbb{R}^{2}$, see for example, \cite{Gof26};
\item[4.] Operators with ``small ellipticity aperture,'' namely,
$\mathfrak{a}_{\mathrm{F}}\defeq \frac{\Lambda}{\lambda}-1\ll 1,$
see \cite{WuNiu23} for the latter class.
\item[5.] Almost linear operators. Precisely, operators $F\in \mathcal{C}^{1}$ such that
\[
|D_{\rm X}F(\rm Z,x)-D_{\rm X}F(\rm Y,x)|\ll 1
\]
for all $x\in B_{1}$ and any symmetric matrices $\mathrm{Z}$ and $\mathrm{Y}$. For this class we refer to \cite{BW21}.
\end{itemize}

We next present several examples that shed light on the role of the exponent $\bar\alpha$, in particular on its optimality and on the limitations of the range of exponents considered in our results.
%\begin{example}
%Under our assumption $(\mathrm{A}_2)$, the parameters $\mu$ and $m$ satisfy $0\leq\mu \leq1-3m$, and consequently, $\bar \alpha \in (2 + \alpha_0,3]$.  The restriction on the exponent $\bar \alpha$, and thus on the values of $\mu$ and $m$, can be illustrated by the following example in the plane $\mathbb{R}^{2}$. Consider the function $u(x)=x_{1}^{3}-x_{1}x_{2}^{2}$ defined in $B_{1}$. Note that $u$ satisfies $|u(0)| = |Du(0)| = |D^2(0)| = 0$, $|D^3u(0)| \neq 0$, and
%\[\operatorname{tr}(\mathfrak{A}(x)D^{2}u)=f(x):=((u^{+})^{\frac{1}{2}}-(u^{-})^{\frac{1}{2}}),\]
%where $\mathfrak{A}$ is defined by\begin{equation*}\mathfrak{A}(x)=
%\begin{cases}
%\mathrm{Id}_{2}, & \mbox{if } x=0,\\[2mm]\mathrm{Id}_{2}+\frac{f(x)}{12|x|^{2}}\begin{pmatrix}x_{1} & -x_{2} \\-x_{2} & -x_{1}\end{pmatrix}, & \mbox{if }x\neq 0.\end{cases}
%\end{equation*}
%Note that $\mathfrak{A}$ is $\mathcal{C}^{0,1/2}$ and uniformly elliptic in $B_{1}$ with the ellipticity constants given by $\lambda=1-\frac{\sqrt{2}}{12}$ and $\Lambda=1+\frac{\sqrt{2}}{12}$. In this case, $\mu = 0, \ m = 1/2, $ and $\bar \alpha = 4$. However, the quotient\[\frac{|u(x)|}{r^{\beta}}=r^{3-\beta} \]is unbounded for $0<r<1$ and $\beta>3$. The Scholium deduced below shows that the estimate fails provided that $|D^3u(0)|\neq 0.$
%\end{example}

\begin{example} Under our assumption $(\mathrm{A}_2)$, the parameters $\mu$ and $m$ satisfy $0\leq\mu \leq1-3m$, and consequently, $\bar \alpha \in (2 + \alpha_0,3]$.  The restriction on the exponent $\bar \alpha$, and thus on the values of $\mu$ and $m$, can be illustrated by the following example in the plane $\mathbb{R}^{2}$. In order to demonstrate the sharpness of the upper bound $\bar\alpha\leq3$ in the borderline regime $\mu=1-3m$, let us  fix $0<m\leq\frac{1}{3}$
and consider the function $
u(x)=x_1^3$ in $B_1$,
satisfying
\[
|u(0)|=|Du(0)|=|D^2u(0)|=0,
\mbox{ and }
|D^3u(0)|\neq0.
\]
Moreover, setting
\[
\mathfrak A=
\begin{pmatrix}
\frac16&0\\
0&1
\end{pmatrix}
\qquad\text{and}\qquad
\mathfrak g(x)=|x_1|^{1-3m},
\]
$\mathfrak A$ is constant and uniformly elliptic, with ellipticity
constants
$\lambda=\frac16,
\
\Lambda=1,$
satisfying
\[
\operatorname{Tr}\big(\mathfrak A D^2u\big)=x_1.
\]
On the other hand,
$(u^+)^m-(u^-)^m
=
\operatorname{sgn}(x_1)|x_1|^{3m},$
and hence
\[
\mathfrak g(x)
\big[(u^+)^m-(u^-)^m\big]
=
|x_1|^{1-3m}
\operatorname{sgn}(x_1)|x_1|^{3m}
=x_1.
\]
Therefore,
\[
\operatorname{Tr}\big(\mathfrak A D^2u\big)
=
\mathfrak g(x)
\big[(u^+)^m-(u^-)^m\big]
\quad \text{in } \ B_1.
\]
Since we are considering the borderline case,
$\mu=1-3m,$ then
$\bar\alpha
=
\frac{2+\mu}{1-m}
=
\frac{3-3m}{1-m}
=3.
$
Finally, 
\[
\frac{\sup_{B_r}|u|}{r^\beta}
=
r^{3-\beta}\to+\infty
\qquad\text{as }r\to0^+
\]
for every $\beta>3=\bar \alpha$. Hence, no growth estimate of order strictly greater than three can hold at a point satisfying
\[
u(0)=|Du(0)|=|D^2u(0)|=0,
\mbox{ and }
|D^3u(0)|\neq0.
\]
This verifies the sharpness of $\bar \alpha$
in the borderline regime $\mu=1-3m$.
\end{example}

\begin{example}
Under our setting, consider also the following profile
\[
u(x):=\operatorname{sgn}(x_1)|x_1|^{\bar\alpha},
\quad x \ \in B_1.
\]
Note that \(u\) is a viscosity solution of
\[
F(x,D^2u)+\mathscr{H}(x,Du)
=
\mathfrak{g}(x)\big((u^+)^m-(u^-)^m\big)
\quad\mbox{in } \ B_1,
\]
where
\[
F(x,\mathrm{X})=\operatorname{Tr}(\mathrm{X}),
\quad \text{and} \quad 
\mathscr{H}(x,\xi):=\langle B,\xi\rangle+\varrho|\xi|^\sigma,
\]
with
$B\equiv b e_1,$ and $
\varrho\equiv\varrho_0,$
where \(b\) and \(\varrho_0\) are non-negative constants, and also
$
\sigma:=\frac{\bar\alpha-2}{\bar\alpha-1}\in(0,1].
$
Moreover, the weight function \(\mathfrak{g}\) is given by
\[
\mathfrak{g}(x)
=
\bar\alpha|x_1|^\mu
\left[
\bar\alpha-1
+b x_1
+\varrho_0\bar\alpha^{\sigma-1}
\operatorname{sgn}(x_1)
\right].
\]
Under the additional condition
$b+\varrho_0\bar\alpha^{\sigma-1}<\bar\alpha-1,$
we have \(\mathfrak{g}\geq0\) in \(B_1\), and
\[
\mathrm{c}_{0}\operatorname{dist}(x,\mathfrak{F})^{\mu}=\mathrm{c}_{0}|x_{1}|^{\mu}\leq \mathfrak{g}(x)\leq \mathrm{C}_0|x_1|^\mu
=
\mathrm{C}_0\operatorname{dist}(x,\mathfrak{F})^\mu,
\]
where $\mathfrak{F}:=\mathfrak{g}^{-1}(0)=\{x_1=0\}$, $\mathrm{c}_{0}=\bar{\alpha}(\bar{\alpha}-1-b-\varrho_{0}\bar{\alpha}^{\sigma-1})$, 
 and $\mathrm{C}_0=\bar\alpha(\bar\alpha-1+b+\varrho_{0}\bar{\alpha}^{\sigma-1}).$
In this case,
\[
\{u>0\}=\{x_1>0\}\cap B_1,
\quad
\{u<0\}=\{x_1<0\}\cap B_1,
\]
and, since \(\bar\alpha>2\),
$Du=D^2u=0
$ on $ \{x_1=0\}\cap B_1.$
Consequently, we deduce
\[
\Gamma_2(u)
=
\{x_1=0\}\cap B_1
=
\mathcal F\cap B_1.
\]
Finally, note that
$
|u(x)|
=
|x_1|^{\bar\alpha}
=
\operatorname{dist}(x,\mathcal F)^{\bar\alpha}, 
$ and that $u\in \mathcal{C}^{2,\beta}(B_1)$, for every $ 0< \beta<\bar\alpha-2.$
\end{example}

Although our examples verify that Theorem \ref{Thm:branching-elliptic} is valid for $\beta$ in a range of $(2+ \alpha_0,3]$, the same argument, under suitable H\"{o}lder assumptions on the data, yields higher-order generalizations of this result for greater values of $\beta$. In fact, for $k\in \mathbb{N}, $ $k\geq 3$, define
\begin{equation*}\mathcal{G}_{k}^{\alpha_0} := \left\{ u \in \mathcal{C}^{0}(B_1) \text{ solving \eqref{1.1} }  \ : \ \|u\|_{\mathcal{C}^{k,\alpha_0}(B_{1/2})} \le \mathrm{K}_0(\mathrm{M}_0, \text{data}) \right\}
\end{equation*}
and
\begin{equation*}
\Gamma_k(u) := \{\vert{}D^iu\vert{}  = 0, \mbox{ for } i=1,\ldots,k\} \cap \partial\{u > 0\} \cap \partial\{u < 0\},
\end{equation*}
the proof of Theorem~\ref{Thm:branching-elliptic} can be adapted directly. Specifically, we have the following Scholium:
\begin{scholium}
Assume that the structural hypotheses $(\mathrm{A}_1)-(\mathrm{A}_4)$ hold with suitable higher H\"{o}lder continuous data. For any $u \in \mathcal{G}_{k}^{\alpha_0}$, assuming $\beta \in \left(k + \alpha_0, \frac{2+\mu}{1-m}\right]$, and $\sigma\in\left[\tfrac{\beta-2}{\beta-1},1\right)\cup(1,2]$, and $x_0 \in \Gamma_k(u) \cap \mathfrak{g}^{-1}(0) \cap B_{1/2}$, then
\[
\sup_{B_r(x_0)}|u|\leq \mathrm{C}r^\beta
\]
for some universal constant $\mathrm{C}>0$.
\end{scholium}

\subsection{Classical foundations and recent advances}

\subsubsection*{Two-phase models: classical results and recent developments}

 Our results are related to the regularity theory for two-phase Alt--Phillips and dead-core problems, as well as to recent developments in higher-order regularity for fully nonlinear elliptic equations. In this subsection, we briefly recall some relevant regularity results that inspire our study. 

Two-phase problems naturally arise in fluid dynamics, jet flows, and cavity models~\cite{Gure66, Tei-Book}. Recent applications include quenching phenomena~\cite{DipKar18, LindPetr08}:$$\Delta u = p\left(\lambda_{+}(u_+)^{p-1}\chi_{\{u>0\}} - \lambda_{-}(u_{-})^{p-1}\chi_{\{u<0\}}\right) \quad \text{in } \Omega \subset \mathbb{R}^n, \quad p \in (0,1),$$obstacle-type problems~\cite{EdquLind09, LSE2009}:$$\Delta u = \lambda_1(x)\chi_{\{u>0\}} - \lambda_2(x)\chi_{\{u<0\}} \quad \text{in } B_1, \qquad \left(\mathrm{M} \ge \sup \lambda_i \ge \inf \lambda_i \ge \frac{1}{\mathrm{M}} > 0\right),$$and dead-core problems~\cite{BdaSSS26, PraRafUrb25}:
$$
\Delta u = (u^{+})^{\gamma} - (u^{-})^{\gamma} \quad \text{in } B_1, \quad \gamma \in (0,1).
$$
While these variational settings allow classical techniques, sign-changing solutions present major technical challenges, as standard tools, such as Harnack inequalities, improvement-of-flatness, and the Strong Maximum Principle, generally fail.

The literature on quasilinear two-phase problems remains sparse, primarily due to the absence of key tools such as Laplacian monotonicity formulas. For the semilinear model
\begin{equation}\label{eqsemilinear}
\Delta u = \lambda_{+}(x)(u^{+})^{q-1} - \lambda_{-}(x)(u^{-})^{q-1} \quad \text{in } B_1,
\end{equation}
Fotouhi and Shahgholian~\cite{Fot-Shah2017} established local regularity for both solutions and free boundaries $\partial\{\pm u > 0\}$, assuming Lipschitz continuous weights $\lambda{\pm}$ and $1 < q < 2$.

In a related vein, Soave and Terracini~\cite{SoaTer18} investigated nodal sets for a particular instance of \eqref{eqsemilinear} with constant positive $\lambda_{\pm}$ and $1 \le q < 2$. They established the finiteness of the vanishing order at every point, characterized the complete spectrum of vanishing orders, proved a weak non-degeneracy property, and obtained regularity and partial stratification results for the nodal set. For the same model, Aghajani~\cite{Agh14} demonstrated that solutions are classical, derived several growth estimates, and characterized free boundary points in terms of homogeneous harmonic polynomials, building upon a foundational result of Caffarelli and Friedman.

\medskip

In the non-variational setting, Bessa and Da Silva \cite{BdaSSS26} recently investigated regularity properties for degenerate quasilinear equations in non-divergence form with Hamiltonian terms:$$\mathcal{Q}_{p, \theta} [u] := \vert{}Du\vert{}^{\theta}\Delta_{p}^{\mathrm{N}}u + \mathscr{H}(Du,x) = \mathscr{F}(u^{+},u^{-},x) \quad \text{in } \Omega \subset \mathbb{R}^{n},$$where $\theta \geq 0$ models the operator's degeneracy. Here, $\Delta_{p}^{\mathrm{N}}$ denotes the normalized $p$-Laplacian, an operator arising in stochastic game theory with strong anisotropic features. The Hamiltonian term $\mathscr{H}$ is assumed to have linear--sublinear growth with respect to the gradient, combining drift and lower-order nonlinear effects (see \eqref{1.2}), while the source term $\mathscr{F}$ incorporates concave-type nonlinearities acting separately on the positive and negative parts of the solution (see \eqref{Assump-RHS}).

{Finally, we should also cite the work of dos Prazeres et al.~\cite{PraRafUrb25}, who investigated the two-phase nonlocal problem
\begin{equation}\label{fracop}
-(-\Delta)^{s}u = (u^{+})^{\gamma} - (u^{-})^{\gamma} \quad \text{in } \mathrm{B}_1,
\end{equation}
where $0 < \gamma < \tfrac{1}{3}$ and $s > 1 - \tfrac{\gamma}{2}$, and established enhanced regularity at branching points. They further examined the integro-differential equation
$$
F(D^{2s}u) = (u^{+})^{\gamma} - (u^{-})^{\gamma} \quad \text{in } \mathrm{B}_1,
$$
where $F$ is fully nonlinear elliptic and $D^{2s}u$ is given by
$$
(D^{2s}u(x))_{ij}=\int_{\mathbb{R}^{n}}\frac{[u(x + y) + u(x - y)- 2u(x)]\langle e_{i},y\rangle\langle e_{j},y\rangle}{|y|^{n+2s+2}},dy.
$$}

{In this setting, they likewise obtained enhanced regularity at branching points. Furthermore, by means of a stability argument as $s \to 1^{-}$, they recovered improved estimates for fully nonlinear elliptic one-phase dead-core problems previously established by Teixeira~\cite{Tei16}.}

These previous contributions motivate our study as the interplay between the two-phase nodal structure, a vanishing spatial weight, and gradient-dependent lower-order terms requires a new approach. Moreover, because the forcing weight can degenerate or vanish on lower-dimensional sets, additional analytical challenges arise. This framework captures a broad class of degenerate elliptic models with nonstandard growth, where degeneracy, gradient dependence, and weighted sources jointly drive the regularity theory.

\subsubsection*{Higher estimates for second-order fully nonlinear PDEs}

In the mid-20th century, Nirenberg~\cite{Nirenberg1953} established H\"older estimates for the Hessian of classical solutions to certain nonlinear elliptic equations in two dimensions. In higher dimensions, Evans~\cite{Evans1982} and Krylov~\cite{Krylov1982} independently proved interior $\mathcal{C}^{2,\beta}$ regularity for viscosity solutions to convex or concave uniformly elliptic equations.

Subsequently, Caffarelli introduced Schauder estimates for inhomogeneous fully nonlinear equations via perturbation and compactness arguments. Specifically, if $F:\operatorname{Sym}(n)\to\mathbb{R}$ is convex, uniformly elliptic, and $F({\bf O_n})=0$, then viscosity solutions to$$F(D^2u)=f(x)\quad\text{in }B_1,$$with $f\in \mathcal{C}^{0,\alpha}(B_1)$, satisfy $u\in \mathcal{C}^{2,\min\{\alpha_0,\alpha\}}(B_{1/2})$ with$$\Vert{}u\Vert{}_{\mathcal{C}^{2,\min\{\alpha_0,\alpha\}}(B_{1/2})} \leq \mathrm{C}\left(\Vert{}u\Vert{}_{L^\infty(B_1)} +\Vert{}f\Vert{}_{\mathcal{C}^{0,\alpha}(B_1)}\right),$$where $\alpha_0\in(0,1)$ depends only on $n,\lambda,\Lambda$. %Thus, the operator convexity is essential for obtaining $\mathcal{C}^{2,\alpha}$ regularity.
Later, Huang~\cite{Huang2002} showed that if $F \in \mathcal{C}^1$ satisfies a Liouville property, any $\mathcal{C}^{1,1}$ viscosity solution to $F(D^2u)=0$ belongs to $\mathcal{C}^{2,\alpha}$ for all $\alpha \in (0,1)$. We must also highlight that Cabré and Caffarelli~\cite{CabreCaff2003} established local $\mathcal{C}^{2,\alpha}$ estimates under conditions weaker than convexity or concavity. Furthermore, Monneau~\cite[Proposition 9.1]{Monneau2009} established pointwise \(\mathcal{C}^{2,\alpha}\) estimates in the \(L^p\) setting for solutions to $F(D^2u)=0$, assuming $F \in \mathcal{C}^2$ alongside a pointwise $\mathcal{C}^2$-Dini condition, bypassing any convexity assumptions.

In a more general setting, for nearly two decades, the existence of a universal $\mathcal{C}^2$ \textit{a priori} theory for arbitrary fully nonlinear elliptic operators remained open, until Nadirashvili and Vlăduţ~\cite{NV2007,NV2008} constructed counterexamples to $\mathcal{C}^{1,1}$ regularity. This motivated the search for structural or qualitative assumptions on $F$ and $u$ that guarantee higher regularity. A substantial advance came from Savin~\cite{Savin2007}, who proved interior $\mathcal{C}^{2,\alpha}$ estimates for \textit{flat} viscosity solutions to $F(D^2u,Du,u,x)=0$, assuming $F$ is uniformly elliptic and smooth near $(0,0,0,x)$. Later, dos Prazeres and Teixeira~\cite{dosPrazTei2016} extended this flat regularity framework to non-convex equations $F(D^2u,x)=f(x)$ with $\mathcal{C}^{0,\alpha}$ data, showing that if $\Vert{}u\Vert{}_{L^\infty(B_1)} \le \delta_0$ for a small universal constant $\delta_0 > 0$, then $u \in \mathcal{C}^{2,\alpha}(B_{1/2})$ with $\Vert{}u\Vert{}_{\mathcal{C}^{2,\alpha}(B_{1/2})} \le \mathrm{K}_0\delta_0.$ Thus, higher regularity can be recovered for non-convex operators under H\"older continuity of the data and a smallness condition.

Furthermore, Bhattacharya and Warren~\cite{BW21}  established explicit interior $\mathcal{C}^{2,\alpha}$ estimates for viscosity solutions to fully nonlinear equations whose operators are \textit{almost linear}, satisfying
$$
|D_XF(\mathrm{M})-D_XF(\mathrm{N})|\leq\varepsilon \qquad \text{for all } \mathrm{M},\mathrm{N}\in\operatorname{Sym}(n).
$$
For $f\in \mathcal{C}^{0,\alpha}(B_1)$, they obtained$$\Vert{}u\Vert{}_{\mathcal{C}^{2,\alpha}(B_{1/2})} \le \mathrm{C}\left(\Vert{}u\Vert{}_{L^\infty(B_1)} + \Vert{}f\Vert{}_{\mathcal{C}^{0,\alpha}(B_1)}\right).$$ More recently, Goffi~\cite{Gof24} proved the same estimate {as above}, for a non-convex $F:\operatorname{Sym}(2)\to\mathbb{R}$, uniformly elliptic,   generalizing Nirenberg's 2D result to viscosity solutions.

Regarding boundary regularity, Krylov~\cite{Krylov1982,Krylov1983} established existence, uniqueness, and $\mathcal{C}^{2,\alpha}$ regularity up to the boundary for fully nonlinear elliptic problems with positively homogeneous operators (such as $F=\inf_j F_j$), applying in particular to Hamilton--Jacobi--Bellman and Monge--Ampère equations. Safonov~\cite{Safonov1989} extended this framework to general Dirichlet problems $$F(D^2u,Du,u,x)=0\quad \mbox{ in } \Omega \quad \mbox{with } u=g \quad \mbox{ on } \partial\Omega,$$ obtaining $\mathcal{C}^{2,\alpha}(\overline{\Omega})$ solutions for $\mathcal{C}^{2,\alpha}$ boundary data and domains. 
Subsequently, Silvestre and Sirakov~\cite{SilSir2014} proved boundary $\mathcal{C}^{2,\alpha}$-type estimates for non-convex, non-concave equations $F(D^2u,Du,x)=f(x)$. Using sup-convolutions and suitable structural conditions, they showed that the boundary gradient and Hessian admit \(\mathcal{C}^{0,\alpha}\) representatives. In particular, near each $x_0 \in \partial\Omega$, the solutions admit the pointwise boundary expansion $$u(x) = u(x_0) + \mathscr{G}(x_0)\cdot(x-x_0) + \frac12\mathbb{H}(x_0)(x-x_0)\cdot(x-x_0) + O(\vert{}x-x_0\vert{}^{2+\alpha}),$$  where $\mathscr{G} \in \mathcal{C}^{0,\alpha}(\partial\Omega,\mathbb{R}^n)$ and $\mathbb{H} \in \mathcal{C}^{0,\alpha}(\partial\Omega,\operatorname{Sym}(n))$, confirming boundary second-order regularity  in the absence of convexity or concavity.

%%%%%%%%%%%%%%%%%%%%%%%%%%%%%%%%%%%%%%%%%%%%%%%%%%%%%%%%%%

\paragraph{Organization of the paper.}

The remainder of the paper is organized as follows. In Section~\ref{Sec2}, we collect the notation, definitions, and auxiliary results that will be used throughout the paper. In Section~\ref{Sec3}, we prove the improved growth estimate at higher-order singular nodal points, which is the main result of the paper. In Section~\ref{Sec4}, we establish several consequences and further applications of this estimate, including sharp gradient growth, $L^p$-average estimates for the Hessian, and a Liouville-type result. Section~\ref{Sec5} is devoted to boundary estimates at higher-order singular nodal points. Finally, in Section~\ref{Sec6}, we investigate further properties of two-phase solutions, proving flipping estimates and a non-degeneracy result.

\section{Preliminaries}\label{Sec2}

The following definition introduces viscosity solutions:

\begin{definition}[{\bf Viscosity Solutions \cite[Remark 2.2]{daSN2021}}]
A continuous function $u \in \mathcal{C}^0(B_1)$ is said to be a \textit{viscosity subsolution} to \eqref{1.1} in $B_1$ if for any $x_{0}\in B_{1}$ and $\varphi\in \mathcal{C}^{2}(B_{1})$ such that $u - \varphi$ has a local maximum at $x_0$, the following inequality holds:
\[
F( D^2 \varphi(x_0),x_0) + \mathscr{H}(x_0,D\varphi(x_{0}))\leq \mathscr{F}(x_0,u^{+}(x_0),u^{-}(x_0)).
\]
Similarly, $u$ is called a \textit{viscosity supersolution} to \eqref{1.1} in $B_1$ if for any $x_{0}\in B_{1}$ and $\varphi\in \mathcal{C}^{2}(B_{1})$ such that $u - \varphi$ has a local minimum at $x_{0}$, one has:
\[
F(D^2 \varphi(x_0),x_0) + \mathscr{H}(x_0,D\varphi(x_{0}))\geq \mathscr{F}(x_0,u^{+}(x_0),u^{-}(x_0)).
\]
We say that $u$ is a \textit{viscosity solution} to \eqref{1.1} if it is both a viscosity subsolution and supersolution.
\end{definition}

We now state an interior regularity result, which provides a H\"{o}lder estimate for the gradient of viscosity solutions under suitable structural and integrability conditions.
%\hspace{1cm}
\medskip
\begin{theorem}[{\bf H\"{o}lder Gradient Estimate \cite[Theorem 1.6]{daSN2021}}]\label{C1alpharegularity}
Assume $(\mathrm{A}_1)-(\mathrm{A}_3)$. Let $f \in L^p(\Omega)\cap \mathcal{C}^0(\Omega)$, where $p > n$, and $\Omega \subset \mathbb{R}^n$ be a bounded domain. Let $u$ be a bounded viscosity solution of 
$$
F(D^{2}u,x) +  \langle \mathcal{B}(x), Du\rangle = f(x) \quad \text{in} \quad B_1 \subset \mathbb{R}^n.
$$ Then, there exists $\alpha \in (0,1)$ and $\theta = \theta(\alpha)$, depending on $n, p, \lambda, \Lambda$, and $\|\mathcal{B}\|_{L^{\infty}(B_1)}$ such that if 
$$
\left(\intav{B_r({0})} \left(\Phi_F(x)\right)^pdx\right)^{1/p}\leq \theta
$$
 holds for all $r \leq \min\{r_0(\theta), \operatorname{dist}(\Omega', \partial\Omega)\}$, for some $r_0 = r_0(\theta) > 0$, then $u \in \mathcal{C}^{1,\alpha}_{\text{loc}}(\Omega)$, and
\begin{equation*} \label{eq:1.9}
    \|u\|_{\mathcal{C}^{1,\alpha}(\overline{\Omega'})} \leq \mathrm{C}\, (\left\| u\right\|_{L^{\infty }(\Omega )} + \|f\|_{L^p(\Omega)}), 
\end{equation*}
for any subdomain $\Omega' \subset\subset \Omega$, where $\mathrm{C}>0$ depends only on $r_0, n, p, \lambda, \Lambda, \alpha, \operatorname{diam}(\Omega), \|\mathcal{B}\|_{L^{\infty}(B_1)}$, and $\operatorname{dist}(\Omega', \partial\Omega)$.
\end{theorem}

In the context of Sobolev spaces, $W^{2,p}$ estimates are available for solutions of \eqref{1.1} with $p>n$ when the operator $F$ is convex or concave.

\begin{theorem}[{\bf $W^{2,p}$ estimates \cite[Theorem 6.7]{daSN2021}}]\label{W2pest}
Let $u$ be a viscosity solution of
\[
F(x,D^{2}u)+\mathscr{H}(x,Du)=f(x)\quad \text{in} \quad B_{1}\subset\mathbb{R}^{n}.
\]
Assume the structural hypotheses $\rm (A_{1})-(A_{3})$ hold and that $F$ is either convex or concave. If
$f\in L^{p}(B_{1})\cap \mathcal{C}^{0}(B_{1})$ for some $n<p<\infty$, then
$u\in W^{2,p}(B_{1/2})$ and the following estimates hold:
\begin{itemize}
\item[(i)] If $\sigma\in(0,1)$, then
\[
\|u\|_{W^{2,p}(B_{1/2})}\leq \mathrm{C}\left(\|u\|_{L^{\infty}(B_{1})} +\|f\|_{L^{p}(B_{1})}+\|\varrho\|_{L^{\infty}(B_{1})}^{\frac{1}{1-\sigma}}\right),
\]
where $\mathrm{C}>0$ depends only on $n$, $\lambda$, $\Lambda$, $p$, $\sigma$, and $\|\mathcal{B}\|_{L^{\infty}(B_{1};\mathbb{R}^{n})}$.

\item[(ii)] If $\sigma\in(1,2]$, then
\[
\|u\|_{W^{2,p}(B_{1/2})}\leq \mathrm{C}\left(\|u\|_{L^{\infty}(B_{1})}+\|f\|_{L^{p}(B_{1})}\right),
\]
where $\mathrm{C}>0$ depends only on $n$, $\lambda$, $\Lambda$, $p$, $\sigma$, $\|\mathcal{B}\|_{L^{\infty}(B_{1};\mathbb{R}^{n})}$, and additionally on
$\|\varrho\|_{L^{\infty}(B_{1})}$.
\end{itemize}
\end{theorem}

Now, recall that viscosity solutions satisfy the following stability property (see, for instance, \cite[Theorem 3.8]{CCKS}).

\begin{proposition}\label{Stability-Prop}
Let $(F_k)_{k \in \mathbb{N}}$ be a sequence of fully nonlinear elliptic operators with ellipticity constants $0< \lambda \leq \Lambda$. Let $(u_k)_{k \in \mathbb{N}} \subset \mathcal{C}^0(\Omega)$ be viscosity solutions to
\[
F_k(D^2 u_k, x) +\langle \mathcal{B}_k(x),Du_k\rangle + \varrho_k(x)|Du_k|^{\sigma}= f_k(x) \quad \text{in } \Omega,
\]
where $(f_k)_{k \in \mathbb{N}}$ is a sequence of continuous functions and $(\mathcal{B}_k)_{k \in \mathbb{N}}$ is a sequence of continuous vector fields. Suppose further that $F_k \to F$ locally uniformly in $\text{Sym}(n)$, and that $u_k\to u$ locally uniformly in $\Omega$, $\mathcal{B}_k\to \mathcal{B}$, $\varrho_ k \to \varrho$, and $f_k \to f$ a.e. in $\Omega$. Then,
\[
F(D^2 u, x) +\langle \mathcal{B}(x),Du\rangle + \varrho(x)|Du|^{\sigma} = f(x) \quad \text{in } \Omega
\]
in the viscosity sense.
\end{proposition}

%In the sequel, we will address a Liouville-type result useful for our purposes. 

%\begin{theorem}[{\bf Liouville-type result}]\label{thm:liouville2}
%Let $u \in \mathcal{C}^0(\mathbb{R}^n)$ be a viscosity solution to
%$$F(D^2 u) = 0 \quad \text{in} \quad \mathbb{R}^n.$$
%There is an exponent $\gamma>0$ depending only on $\lambda, \Lambda$ and $n$ such that if
%$$\displaystyle \lim_{|x| \to +\infty} \frac{|u(x)|}{|x|^{2+\gamma}}=0,$$
%then $u$ is a (quadratic) function.
%\end{theorem}

%\textcolor{red}{We need to address a proof for such a Liouville-type result}.

We will need the following comparison principle, whose proof is analogous to that in \cite[Theorem 1.7]{BBdaSSS25}. For this reason, we omit it here.

\begin{theorem}[{\bf Comparison Principle}]\label{ComPri}
 Let $\Omega\subset \mathbb R^n$ be  a bounded domain  and let $\mathfrak{c}, f_{1}, f_{2} \in \mathcal{C}^0(\overline{\Omega})$, and let $\mathcal{F} : \mathbb{R} \to \mathbb{R}$ be a continuous and non-decreasing function such that $\mathcal{F}(0) = 0$. Suppose that $u, v \in \mathcal{C}^{0}(\Omega)$ are functions satisfying
\begin{equation*}\label{CP-General}
\left\{
\begin{array}{rcll}
F(x,D^{2}u)+\mathscr{H}(x,Du) + \mathfrak{c}(x) \mathcal{F}(u) & \geq & f_1(x) & \text{in } \Omega, \\[0.2cm]
F(x,D^{2}v)+\mathscr{H}(x,Dv)+ \mathfrak{c}(x) \mathcal{F}(v) & \leq & f_2(x) & \text{in } \Omega,
\end{array}
\right.
\end{equation*}
in the viscosity sense, where the operator $F$ satisfies ${(\mathrm{A}_1)}$   and $\mathscr{H}$ is a Hamiltonian function defined in \eqref{1.2} with $\mathcal{B}$ and $\varrho$ satisfying the assumption ${(\mathrm{A}_2)}$. If, $\mathfrak{c} \leq 0$, $f_1 < f_2$, 
and $u\leq v$ on $\partial \Omega$, then, $u \leq v$ in $\Omega$.
\end{theorem}

\section{Improved estimates at higher-order singular nodal points}\label{Sec3}

This section is devoted to handling  the proof of the improved growth estimate at points of the set $\Gamma_2(u)$, Theorem \ref{Thm:branching-elliptic}.

\begin{proof}[{\bf Proof of Theorem \ref{Thm:branching-elliptic}}]
Initially, if $u\equiv 0$, we observe that the result is trivial. In this case, we can suppose that $\|u\|_{L^{\infty}(B_{1})}>0$. Now, as $u\in \mathcal{G}_{2}^{\alpha_0}$ it follows that 
\begin{align}\label{the1:gradient}
\|u\|_{\mathcal{C}^{2,\alpha_0}(B_{{1}/{2}})}\leq \mathrm{K}_0.\end{align}
Without loss generality, we assume that $x_0:=0\in \Gamma_2(u)\cap \mathfrak{g}^{-1}(0)\cap B_{1/2}$ and $\|u\|_{L^{\infty}(B_{1})}=1$. It suffices to prove that \begin{align} \label{the1:2}
|u(x)|\leq \mathrm{C}|x|^{\beta}.
\end{align}
In order to establish \eqref{the1:2}, we argue by contradiction. Assume that there exist sequences of $u_j\in\mathcal{G}_{2}^{\alpha_0}$, $\mathcal{B}_j\in \mathcal{C}^{0,\alpha_0}(B_1;\mathbb{R}^n)\cap L^{\infty}(B_1;\mathbb{R}^n)$,
$\varrho_j,\mathfrak{g}_{j}\in \mathcal{C}^{0,\alpha_0}(B_1)\cap L^{\infty}(B_1)$,
and $x_j \in B_{1/2}$ such that
$$F(x,D^2u_j)+\langle B_j(x),Du_j\rangle+\varrho_j(x)|Du_j|^\sigma=\mathfrak{g}_{j}(x)\left((u_j^+)^m-(u_j^-)^m\right)\quad \text{in} \quad B_1,$$
 with $\|u_j\|_{\mathcal{C}^{2,\alpha_0}(B_{1/2})}\leq \mathrm{K}_0$, $0\in \Gamma_2(u_j)\cap\mathfrak{g}_{j}^{-1}(0)\cap B_{1/2}$, and \begin{align}\label{the1:3}
 |u_j (x_j )|>j|x_j |^{\beta}.\end{align}
For each $j\in \mathbb{R}$, define $$\phi_j(\rho):=\displaystyle\sup_{\rho<r<1/2}r^{2+\alpha_0-\beta}[D^2u_j]_{\alpha_0,B_r},$$
then, we have
$$\displaystyle\lim_{\rho\rightarrow0}\phi_j(\rho)=\displaystyle\sup_{0<r<1/2}r^{2+\alpha_0-\beta}[D^2u_j]_{\alpha_0,B_r}.$$
Since $0\in \Gamma_2(u_j)\cap B_{1/2}$, it follows that $u_j(0)=|Du_j(0)|=|D^2u_j(0)|=0$, which together with  \eqref{the1:3} leads to
\begin{align}\label{the1:4}
    j&<\frac{|u_j(x_j)|}{|x_j|^{\beta}}=\frac{\left|u_j(x_j)-u_j(0)-Du_j(0)\cdot x_j-D^2u_j(0)\cdot x_j^2\right|}{|x_j|^{\beta}}\notag\\
    &\leq\frac{[D^2u_j]_{\alpha_0,B_{|x_j|}}|x_j|^{2+\alpha_0}}{|x_j|^{\beta}}
    \leq \phi_j\left(\frac{|x_j|}{2}\right) \leq \displaystyle\lim_{\rho\rightarrow0}\phi_j(\rho)=+\infty.
\end{align}
Therefore, by the definition of $\phi_j$, for $j\geq 3$, there exists $\rho_j   \in\left(\frac{1}{j},\frac{1}{2}\right)$ such that 
\begin{align}\label{the1:5}
    {\rho_j}^{2+\alpha_0-\beta}[D^2u_j]_{\alpha_0,B_{\rho_j}}\geq \frac{1}{2}\phi_j\left(\frac{1}{j}\right)\geq \frac{1}{2}\phi_j\left(\rho_j\right).
\end{align}
By virtue of \eqref{the1:gradient}, \eqref{the1:4}, and \eqref{the1:5}, we deduce that
\begin{equation}\label{Eq3.6}
    {\rho_j}^{\beta-(2+\alpha_0)}\leq \frac{2[D^2u_j]_{\alpha_0,B_{\rho_j}}}{\phi_j({1}/{j})}\leq \frac{2\mathrm{K}_0}{\phi_j({1}/{j})}\leq \frac{2\mathrm{K}_0}{\phi_j(\rho_j)}\quad \text{as} \quad j\rightarrow \infty.
\end{equation}
Note that $2+\alpha_0<\beta$, which implies that $\rho_j\rightarrow 0$ as $j\rightarrow \infty$.

Define now the blow up sequence $$w_j(x):=\frac{u_j(\rho_jx)}{\rho_j^{\beta}\phi_j(\rho_j)}, \quad \forall \ x \in B_{\frac{1}{2\rho_j}}.$$
Then, $w_j$  satisfies, in the viscosity sense, the following equation
$$\widetilde{F}_j(x,D^2w_j)+\langle\widetilde{\mathcal{B}}_j(x),Dw_j\rangle+\widetilde{\varrho_j}(x)|Dw_j|^{\sigma}=\widetilde{\mathscr{F}}_j(x,w_j^+,w_j^-) \quad \text{in} \quad B_{\frac{1}{2\rho_j}},$$
where
$\widetilde{F}_j(x,D^2w_j):=\frac{\rho_j^{2-\beta}}{\phi_j(\rho_j)}F\left(\rho_jx,\frac{\phi_j(\rho_j)}{\rho_j^{2-\beta}}D^2w_j\right)$, $$\widetilde{\mathcal{B}}_j(x):=\rho_j\mathcal{B}_j(\rho_jx), \ 
\widetilde{\varrho_j}(x):=\frac{\rho_j^{2-\beta-\sigma+\sigma\beta}}{\phi_j(\rho_j)^{1-\sigma}}\varrho_j{ (\rho_jx)}, \quad \mbox{and} $$
$$\widetilde{\mathscr{F}}_j(x,w_j^+,w_j^-):=\frac{\rho_j^{2-\beta(1- m)}}{\phi_j(\rho_j)^{1-m}}\mathfrak{g}_{j}(\rho_{j}x)\left((w_j^+)^m-(w_j^-)^m\right).$$
Observe that $w_j(0)=|Dw_j(0)|=|D^2w_j(0)|=0$, and by \eqref{the1:5}, we have
\begin{align}\label{the1:6}
    [D^2w_j]_{\alpha_0,B_1}
    =  [D^2u_j]_{\alpha_0,B_{\rho_j}}\rho_j^{2+\alpha_0-\beta}\cdot \frac{1}{\phi_j(\rho_j)}
    \geq  \frac{\phi_j(1/j)}{2\phi_j(1/j)}=\frac{1}{2}.
\end{align}
 For any fixed $r>1$, there exists $j_0$ such that, for $j\geq j_0$, it holds that $B_r \subsetneq B_ {1/(2\rho_j )} $. For every $x\in B_r$, we deduce by $\rho_jr\in\left(\rho_j,\frac{1}{2}\right)$ that 
 \begin{align}\label{the1:7}
     |w_j(x)|\leq \frac{[D^2u_j]_{\alpha_0,B_{\rho_jr}}(\rho_jr)^{\alpha_0+2-\beta}}{\phi_j(\rho_j)}r^{\beta}\leq \frac{{\phi_j(\rho_jr)}}{{\phi_j(\rho_j)}}r^{\beta}\leq r^{\beta}.
 \end{align}
 Consequently, we obtain $\|w_j\|_{L^{\infty}(B_r)}\leq r^{\beta}$.
 Similarly, we have 
 $$|Dw_j(x)|\leq r^{\beta-1},\quad |D^2w_j(x)|\leq r^{\beta-2},  \quad [D^2w_j(x)]_{\alpha,B_r}\leq r^{\beta-(2+\alpha_0)}.$$
Hence, since $r>1$, we get $$\|w_j\|_{\mathcal{C}^{2,\alpha_0}(B_r)}\leq \|w_j\|_{L^\infty(B_r)}+\|Dw_j\|_{L^\infty(B_r)}+\|D^2w_j\|_{L^\infty(B_r)}+[D^2w_j(x)]_{\alpha_0,B_r}\leq 4r^{\beta}.$$
Thus, by applying Arzel\'a-Ascoli theorem, there exists some $w_\infty \in C_{\rm loc}^{2,\alpha_0}(B_r)$ such that,  up to a subsequence, it yields $w_j\rightarrow w_\infty$ in $C_{\rm loc}^{2,\alpha_0}(B_r)$. In addition,   $$w_\infty(0)=|Dw_\infty(0)|=|D^2w_\infty(0)|=0.$$
From \eqref{the1:6}, we obtain $[D^2w_\infty]_{\alpha_0,B_1}\geq\frac{1}{2}$.
Moreover, by \eqref{the1:7}, it follows that $w_\infty(x)\leq |x|^{\beta}$.

Now, letting $j\rightarrow\infty$, we infer from $\rho_j\rightarrow0$  that
$$
\|\widetilde{\mathcal{B}}_j(x)\|_{L^{\infty}(B_{1};\mathbb{R}^{n})}=\rho_j\|\mathcal{B}_j(\rho_j\cdot)\|_{L^{\infty}(B_{1};\mathbb{R}^{n})}\leq \rho_j\|\mathcal{B}_j\|_{L^\infty(B_{1/2};\mathbb{R}^{n})}\rightarrow0.
$$
In order to estimate $\widetilde{\varrho_j}(x)$, we consider two cases separately: $\sigma\in \left[\frac{\beta-2}{\beta-1},1\right)$ and $\sigma\in(1,2]$.

\noindent \boxed{\textbf{Case 1}}: Assume that $\sigma\in \left[\frac{\beta-2}{\beta-1},1\right)$,   then $2-\beta + \sigma(\beta -1) >0$, and as $j\to \infty$, $\rho_j \to 0$, yielding 
\begin{align*}|\widetilde{\varrho_j}(x)|=&\left|\frac{\rho_j^{2-\beta-\sigma+\sigma\beta}}{\phi_j(\rho_j)^{1-\sigma}}\varrho_j{ (\rho_jx)}\right|
\leq \frac{\rho_j^{2-\beta-\sigma+\sigma\beta}}{\phi_j(\rho_j)^{1-\sigma}}\|\varrho_j\|_{{L^\infty(B_{1/2})}}
 \leq \frac{\rho_j^{2-\beta-\sigma+\sigma\beta}}{j^{1-\sigma}}\|\varrho_j\|_{{L^\infty(B_{1/2})}}\rightarrow 0.
 \end{align*}
 
\noindent \boxed{\textbf{Case 2}}: Assume that $\sigma\in(1,2]$, using \eqref{Eq3.6}, as $j\to \infty$, we obtain
\begin{align*}|\widetilde{\varrho_j}(x)|
 \leq {\rho_j^{2-\beta-\sigma+\sigma\beta}}\left(\frac{\rho_j^{\beta-(2+\alpha_0)}}{2\mathrm{K}_0}\right)^{1-\sigma}\|\varrho_j\|_{{L^\infty(B_{1/2})}}
 \leq  \rho_j^{\sigma(1+\alpha_0)-\alpha_0} \frac{\|\varrho_j\|_{{L^\infty(B_{1/2})}}}{(2\mathrm{K}_0)^{1-\sigma}}        \rightarrow 0.
 \end{align*}
Therefore, in both cases, we conclude that $|\widetilde{\varrho_j}(x)|\rightarrow 0$ as $j\rightarrow\infty$.

For the source term $\bar{\mathscr{F}}_{j}$, note that the structural condition ($\rm A_{2}$) combined with $0\in \mathfrak{g}_{j}^{-1}(0)$ {implies} that 
\[
\mathfrak{g}_{j}(\rho_{j}x)\leq \mathrm{C}_{0}\rho_{j}^{\mu}|x|^{\mu}.
\]
Using this fact jointly with $m\in(0,1)$, and $ \beta\in \left(2+\alpha_0,\frac{2+\mu}{1-m}\right]$, by \eqref{the1:7} we have
\begin{eqnarray*}
|\widetilde{\mathscr{F}}_j(x,w_j^+,w_j^-)|&\leq &\mathrm{C}_{0}\left|\frac{\rho_j^{2+\mu-\beta+\beta m}}{\phi_j(\rho_j)^{1-m}}\right|\cdot|x|^\mu\left((w_j^+)^m-(w_j^-)^m\right)\\
    &\leq  &\frac{2\mathrm{C}_{0}\rho_j^{2+\mu-\beta+\beta m}}{j^{1-m}}r^{\mu+{\beta m}}\rightarrow 0.
\end{eqnarray*}
Additionally, from \eqref{Eq3.6}, one has
$$\frac{\phi_j(\rho_j)}{\rho_j^{2-\beta}}\leq 2\mathrm{K}_0\rho_j^{\alpha_0}\rightarrow 0 \quad \text{as} \quad j\rightarrow \infty.$$
Thus, the structural conditions $\rm (A_1)-(A_2)$ and $\rm(A_4)$ yield
\[
\widetilde{F}_j(x,D^2w_j):=\frac{\rho_j^{2-\beta}}{\phi_j(\rho_j)}F\left(\rho_jx,\frac{\phi_j(\rho_j)}{\rho_j^{2-\beta}}D^2w_j\right) \to \text{Tr}(\mathfrak{A}(0)D^2w_\infty),
\]
locally uniformly as $j \to \infty$. 
Therefore, by the stability theorem for viscosity
solutions, see \cite[Section 6]{CIL92}, we conclude that 
\(w_\infty\)  solves
\begin{equation*}
\text{Tr}(\mathfrak{A}(0)D^2w_\infty)=0
\quad \text{in } \quad \mathbb{R}^n.
\end{equation*}
Since $\beta-\frac{2+\mu}{1-m}<0$,  by \eqref{the1:7}, it follows that \begin{equation*}
\frac{|w_\infty(x)|}{|x|^{\frac{2+\mu}{1-m}}} \leq \mathrm{C}|x|^{\beta-\frac{2+\mu}{1-m}}\rightarrow 0 \quad \text{as} \quad {|x| \to +\infty} .    
\end{equation*}
Hence, by the Liouville-type property (see e.g. \cite[Proposition 1.19]{FR-O}), 
 we conclude that $w_\infty$ is  a quadratic polynomial function. Since $w_\infty(0)=|Dw_\infty(0)|=|D^2w_\infty(0)|=0$,
it follows that $w_\infty=0$, which yields a contradiction with \eqref{the1:6}.
\end{proof}

\section{Some consequences of the improved growth estimates} \label{Sec4}

The improved estimates proved above are central to our study of higher-order singular nodal points. In this section, we explore its main consequences for local solution behavior and structure.

\subsection{Sharp gradient growth}

As a consequence of Theorem \ref{Thm:branching-elliptic}, we provide the proof of Corollary \ref{Cor:SharpGradientGrowth} that gives a sharp growth of the gradient {along} the higher-order singular nodal points. The proof proceeds by contradiction, using geometric decay estimates as introduced in \cite{CKS00}.

\begin{proof}[{\bf Proof of Corollary \ref{Cor:SharpGradientGrowth}}]
For simplicity, assume that \(x_{0}=0 \in \Gamma_2(u)\cap B_{1/2}\). It suffices to prove that for all $j\in\mathbb{N}$
\begin{equation}\label{Grad1}
\mathrm{S}_{j+1} \leq \max\left\{\mathrm{C} \, 2^{-\widehat{\beta}(j+1)}, \, 2^{-\widehat{\beta}} \mathrm{S}_j\right\},
\end{equation}
where \(\displaystyle \mathrm{S}_j \coloneqq \sup_{B_{2^{-j}}} |Du|\) and $\widehat{\beta} = \bar{\alpha}-1 = \frac{2+\mu}{1-m} -1$, for a universal constant $\mathrm{C}>0$.
Suppose, by contradiction, that \eqref{Grad1} fails. Then, for each $k \in \mathbb{N}$, there exists $j_k \in \mathbb{N}$ such that
\begin{equation}\label{Div2}
\mathrm{S}_{j_k+1} > \max\left\{k \, 2^{-\widehat{\beta}(j_k+1)}, \, 2^{-\widehat{\beta}} \mathrm{S}_{j_k}\right\}.
\end{equation}
For each $k$, define the auxiliary function
\[
u_k(x) \defeq \frac{2^{j_k} u(2^{-j_k}x)}{\mathrm{S}_{j_k+1}}, \quad x \in B_1.
\]
Thus, applying Theorem \ref{Thm:branching-elliptic} with $\beta=\bar\alpha$ and the inequality \eqref{Div2}, we obtain
\begin{equation}\label{estdetildeuk}
|u_k(x)| \leq \frac{\mathrm{C}\|u\|_{L^{\infty}(B_{1})} 2^{-j_k(1+\widehat{\beta})} 2^{j_k}}{\mathrm{S}_{j_k+1}} \leq \frac{2^{\widehat{\beta}} \mathrm{C}\|u\|_{L^{\infty}(B_{1})}}{k}, \quad x \in B_1.
\end{equation}
Moreover,
\[
\|Du_k\|_{L^\infty(B_{1/2})} = 1, \quad \text{and} \quad \|Du_k\|_{L^\infty(B_1)} \leq 2^{\widehat{\beta}}.
\]
Note also that \(u_k\) solves, in the viscosity sense, the following equation
\begin{align*}
F_k(x, D^2 u_k) + \langle \mathcal{B}_k(x), Du_k \rangle + \varrho_k(x)|Du_k|^{\sigma}
&= \frac{2^{-j_k(1+\mu+m)}}{\mathrm{S}_{j_k+1}^{1-m}} |x|^\mu \big((u_k^+)^m -  (u_k^-)^m\big)\\
&\eqqcolon f_k(x,u_k)
\quad \text{in} \ B_1,
\end{align*}
where \(\mathcal{B}_k(x) = 2^{-j_k} \mathcal{B}(2^{-j_k}x)\).  
By \eqref{Div2}, \eqref{estdetildeuk}, and the definition of $\widehat{\beta}$, one has
\[
\|f_k\|_{L^\infty(B_1)} \leq \frac{2^{1+\widehat{\beta}} \mathrm{C}^{m}}{k},
\]
and, by construction, \(\|\mathcal{B}_k\|_{L^\infty(B_1;\mathbb{R}^n)} \leq \|\mathcal{B}\|_{L^\infty(B_1;\mathbb{R}^n)}\).  
Applying the uniform gradient estimate (Theorem~\ref{C1alpharegularity}), together with \eqref{estdetildeuk} and the above bounds, we arrive at
\begin{align*}
1 &= \|Du_k\|_{L^\infty(B_{1/2})} \leq \mathrm{C}\left(n, \mu, m, \|\mathcal{B}\|_{L^\infty(B_1;\mathbb{R}^n)}\right) 
\left(\|u_k\|_{L^\infty(B_1)} + \|f_k(\cdot, u_k)\|_{L^\infty(B_1)} \right) \\
&\leq \mathrm{C} 
\left( \frac{2^{\widehat{\beta}}\|u\|_{L^{\infty}(B_{1})}}{k} + \frac{2^{1+\widehat{\beta}}\mathrm{C}^{m-1}}{k}  \right) \to 0,
\end{align*}
a contradiction for \(k \gg 1\). Therefore, the result follows.
\end{proof}

\subsection{$L^p$-average estimates for the Hessian}

We now establish the $L^{p}$-average estimates for the Hessian. The key idea is to combine the improvement estimates from Theorem \ref{Thm:branching-elliptic} with the $W^{2,p}$ estimates provided by Theorem \ref{W2pest} in a scaling argument over balls of radius $r$.

\begin{proof}[{\bf Proof of Theorem \ref{Lpaverage}}] 
Without loss of generality, suppose that $x_{0}=0$. Given $r\in (0,1/2)$, we define the following scaled profile
\begin{equation*}
u_{r}(x)=\frac{u(rx)}{r^{\beta}},\quad x\in B_{1}.    
\end{equation*}
Initially, we have that Theorem \ref{Thm:branching-elliptic} is still valid under the hypotheses of this theorem. Therefore, we immediately have that
\begin{equation}\label{estofur}
\|u_{r}\|_{L^{\infty}(B_{1})}\leq \mathrm{C}\|u\|_{L^{\infty}(B_{1})},
\end{equation}
where $\mathrm{C}$ is a positive constant depending only only on $n$, $\lambda$, $\Lambda$, $m$, $\sigma$,
$\|\mathcal{B}\|_{\mathcal{C}^{0,\alpha_{0}}(B_1;\mathbb{R}^{n})}$,
$\|\varrho\|_{\mathcal{C}^{0,\alpha_{0}}(B_1)}$, and 
$\|\mathfrak g\|_{L^\infty(B_1)}$. Moreover, it is possible to check that $u_{r}$ solves, in the viscosity sense,
\[
F_{r}(x,D^{2}u_{r})+\langle \mathcal{B}_{r}(x), Du_{r}\rangle+\varrho_{r}(x)|Du_{r}|^{\sigma}=\mathscr{F}_{r}(x,u_{r}^{+},u_{r}^{-})\quad \text{in}\quad B_{1},
\]
where $F_{r}(x, \mathrm{X})=r^{2-\beta}F(rx,r^{\beta-2}\mathrm{X})$, $\mathcal{B}_{r}(x)=r\mathcal{B}(rx)$, $\varrho_{r}(x)=r^{\sigma(\beta-1)-(\beta-2)}\varrho(rx)$, and
\[
\mathscr{F}_{r}(x)=\mathscr{F}_{r}(x,u_{r}^{+},u_{r}^{-})=r^{2-\beta(1-m)}\mathfrak{g}(rx)((u_{r}^{+})^{m}-(u_{r}^{-})^{m}).
\] 
In this case, we have that the convexity (or concavity) of $F$ is preserved by the scaling, and hence $F_{r}$ is also convex (or concave). Moreover, $F_{r}$ satisfies ($\rm A_{1}$) with the same ellipticity constants $\lambda$ and $\Lambda$. Regarding condition ($\rm A_{3}$), since $F$ satisfies this condition and $r<1$, a straightforward calculation yields
\[
\frac{|F_{r}(\mathrm{X},x)-F_{r}(\mathrm{X},y)|}
{\|\mathrm{X}\|}\leq \tau(r|x-y|)\leq \tau(|x-y|),
\]
for every $\mathrm{X}\in \mathrm{Sym(n)}$ and every $x,y\in B_{1}$. This proves that $F_{r}$ also satisfies condition ($\rm A_{3}$). In addition, since $r<1,$ it follows that
$$
\|\mathcal{B}_{r}\|_{L^{\infty}(B_{1};\mathbb{R}^{n})}\leq r\|\mathcal{B}\|_{L^{\infty}(B_{r};\mathbb{R}^{n})}\leq \|\mathcal{B}\|_{L^{\infty}(B_{1};\mathbb{R}^{n})},
$$
and besides, in view of $\sigma \geq \frac{\beta-2}{\beta-1}$, we have
$$
\|\varrho_{r}\|_{L^{\infty}(B_{1})}\leq r^{\sigma(\beta-1)-(\beta-2)}\|\varrho\|_{L^{\infty}(B_{r})}\leq \|\varrho\|_{L^{\infty}(B_{1})}.
$$
For the source term $\mathscr{F}_{r}$,  as in the proof of Theorem \ref{Thm:branching-elliptic}, the fact $0\in \mathfrak{g}^{-1}(0)$ combined with the assumption ($\rm A_{2}$) implies that $0\in F$ and $\mathfrak{g}(rx)\leq \mathrm{C}_{0}r^{\mu}|x|^{\mu}$. Hence, this fact and the estimate \eqref{estofur} yield
\[
|\mathscr{F}_{r}(x,u_{r}^{+},u_{r}^{-})|\leq 2\mathrm{C}_{0}\mathrm{C}^{m}\|u\|_{L^{\infty}(B_{1})}^{m}r^{2+\mu-\beta(1-m)}|x|^{\mu}\leq 2\mathrm{C}_{0}\mathrm{C}^{m}\|u\|_{L^{\infty}(B_{1})}^{m},
\]
for all $x\in B_{1}$.
Therefore, we can apply the $W^{2,p}$ estimates (Theorem \ref{W2pest}) and conclude, in particular, that
\begin{equation*}
\|D^{2}u_{r}\|_{L^{p}(B_{1/2})}\leq \mathrm{C}\left(\|u_{r}\|_{L^{\infty}(B_{1})}+\|\mathscr{F}_{r}\|_{L^{p}(B_{1})}+\|\varrho_{r}\|_{L^{\infty}(B_{1})}^{\frac{1}{1-\sigma}}\right),
\end{equation*}
if $\sigma<1$, where $\mathrm{C}>0$ depends only on $n$, $\lambda$, $\Lambda$, $p$, $m$, $\sigma$, $\mathrm{C}_{0}$, and $\|\mathcal{B}\|_{L^{\infty}(B_{1};\mathbb{R}^{n})}$. In the superlinear case, the following estimate holds
\begin{equation*}
\|D^{2}u_{r}\|_{L^{p}(B_{1/2})}\leq \mathrm{C}\left(\|u_{r}\|_{L^{\infty}(B_{1})}+\|\mathscr{F}_{r}\|_{L^{p}(B_{1})}\right),
\end{equation*}
where the constant $\mathrm{C}$ in this estimate depends on the parameters from the previous case, and additionally on $\|\varrho\|_{L^{\infty}(B_{1})}$. Using  \eqref{estofur} and the estimate for $\mathscr{F}_{r}$, in both cases we conclude that
\[
\|D^{2}u_{r}\|_{L^{p}(B_{1/2})}\leq \widetilde{\mathrm{M}}_{0},
\]
where $\widetilde{\mathrm{M}}_{0}>0$ depends only on $n$, $\lambda$, $\Lambda$, $p$, $m$, $\sigma$, $\mathrm{C}_{0}$, $\|\mathcal{B}\|_{\mathcal{C}^{0,\alpha_{0}}(B_{1};\mathbb{R}^{n})}$, $\|\varrho\|_{\mathcal{C}^{0,\alpha_{0}}(B_{1})}$, and $\|u\|_{L^{\infty}(B_{1})}$. Using this inequality, we obtain
\begin{eqnarray*}
\left(\intav{B_{\frac{r}{2}}}|D^{2}u|^{p}\,dx\right)^{\frac{1}{p}}\leq \mathrm{M}_{0}r^{\beta-2},
\end{eqnarray*}
where $\mathrm{M}_{0}={2^{\frac{n}{p}}\widetilde{\mathrm{M}}_{0}}{|B_{1}|^{-\frac{1}{p}}}$. This ends the proof.
\end{proof}

\subsection{A Liouville-type result}

We conclude this section by presenting the proof of Theorem~\ref{thm:liouville}, which establishes a Liouville-type result under the optimal growth condition with exponent $\bar{\alpha}=\frac{2+\mu}{1-m}$.

\begin{proof}[{\bf Proof of Theorem \ref{thm:liouville}}]
As a matter of convenience, we assume $x_{0}=0$. For $\mathrm{R} > 1$, define
\[
u_{\mathrm{R}}(x) := \frac{u(\mathrm{R}x)}{\mathrm{R}^{\bar{\alpha}}}.
\]
Since $u_{\mathrm{R}}(0) = 0$ and $u_{\mathrm{R}}$ solves \eqref{1.1} in the viscosity sense, Theorem~\ref{Thm:branching-elliptic} implies
\begin{equation}\label{5.2}
|u_{\mathrm{R}}(x)| \leq \mathrm{C} \|u_{\mathrm{R}}\|_{L^\infty(\mathbb{R}^n)} |x|^{\bar{\alpha}}.
\end{equation}
If $|\mathrm{R}x|$ remains bounded, then $u(\mathrm{R}x)$ is bounded, hence,
\begin{equation}\label{5.3}
|u_{\mathrm{R}}(x)| \to 0 \quad \text{as } \mathrm{R} \to +\infty.
\end{equation}
The same conclusion holds when $|\mathrm{R}x| \to +\infty$, since \eqref{Hip-Liouville:Thm} yields
\begin{equation}\label{5.4}
|u_{\mathrm{R}}(x)| = \frac{|u(\mathrm{R}x)|}{|\mathrm{R}x|^{\bar{\alpha}}} |x|^{\bar{\alpha}} \to 0, \quad x \neq 0.
\end{equation}
Combining \eqref{5.3} and \eqref{5.4}, we deduce
\begin{equation}\label{5.5}
\|u_{\mathrm{R}}\|_{L^{\infty}(\mathbb{R}^{n})} \to 0 \quad \text{as } \quad \mathrm{R} \to +\infty.
\end{equation}
Arguing by contradiction, assume $u \not\equiv 0$. Then, there exists $z \in \mathbb{R}^n \setminus \{0\}$ such that $|u(z)| > 0$. In view of \eqref{5.5}, we can choose $\mathrm{R}$ sufficiently large so that $z \in B_\mathrm{R}$ and$$\Vert{}u_\mathrm{R}\Vert{}_{L^{\infty}(\mathbb{R}^n)} < \frac{1}{2} \frac{\vert{}u(z)\vert{}}{\vert{}z\vert{}^{\bar{\alpha}}}.$$ Using \eqref{5.2}, and the this last inequality, we deduce
\[
\frac{|u(z)|}{|z|^{\bar{\alpha}}}
\leq \sup_{B_{\mathrm{R}}} \frac{|u(x)|}{|x|^{\bar{\alpha}}}
= \sup_{B_1} \frac{|u_{\mathrm{R}}(x)|}{|x|^{\bar{\alpha}}}
< \frac{1}{2} \frac{|u(z)|}{|z|^{\bar{\alpha}}},
\]
which is a contradiction. Therefore, $u \equiv 0$.
\end{proof}

\section{Boundary estimates at higher-order singular nodal points}\label{Sec5}

In what follows, we establish our boundary estimates at higher-order singular nodal points. The proof follows the general strategy developed in Theorem~\ref{Thm:branching-elliptic}, based on blow-up and compactness arguments. Nevertheless, the boundary and its prescribed data require a more delicate analysis, as rescaled solutions must be controlled up to the boundary. In particular, the blow-up argument must preserve both the boundary condition and the higher-order vanishing of the boundary data at the singular point.

\begin{proof}[{\bf Proof of Theorem \ref{Thm:boundary}}]
First, we note that if $\|u\|_{L^{\infty}(B^{+}_{1})}=0$, then the result is immediate, since, by continuity, $u\equiv 0$. Thus, it suffices to consider the case $\|u\|_{L^{\infty}(B_{1}^{+})}>0$. Now, by $u\in \widetilde{\mathcal{G}}_{2}^{\alpha_{0}}$, it is known that
\begin{equation*}\label{the8:gradient}
\|u\|_{\mathcal{C}^{2,\alpha _0}\left(\overline{B_{{1}/{2}}^+}\right)}\leq \mathrm{K}_0.
\end{equation*}
Without loss generality, we assume that $x_0:=0\in \Gamma_2(u)\cap \mathfrak{g}^{-1}(0) \cap T_{{1}/{2}}$ and $\|u\|_{L^{\infty}(B_{1}^{+})}=1$. It suffices to prove that \begin{align} \label{the8:2}
|u(x)|\leq \mathrm{C}|x|^{\beta}, \quad \forall \ x\in B_{{1}/{2}}^+\cup T_{{1}/{2}}.
\end{align}
In order to establish \eqref{the8:2}, we argue by contradiction. Assume that there exist sequences of $$u_j\in \mathcal{C}^0(B_1^+\cup T _1), \ \mathcal{B}_j\in \mathcal{C}^{0,\alpha_0}(B_1^+;\mathbb{R}^n)\cap L^{\infty}(B_1^+;\mathbb{R}^n), \
\varrho_j,\mathfrak{g}_{j}\in \mathcal{C}^{0,\alpha_0}(B_1^+)\cap L^{\infty}(B_1^+),$$  
$ \mathrm{h}_j \in \mathcal{C}^{2,\widehat\beta -2}{(T _1)},$ and $x_j\in B_{{1}/{2}}^+\cup T_{{1}/{2}}$ such that
\begin{equation*}
\left\{
\begin{array}{rcll}
F(x,D^2u_j)+\langle B_j(x),Du_j\rangle+\varrho_j(x)|Du_j|^\sigma &= &\mathfrak{g}_{j}(x)\left((u_j^+)^m-(u_j^-)^m\right)& \text{in} \quad B_1^+,\\
u_j&=&\mathrm{h}_j &\text{on  }\,\, \ T_{1}.
\end{array}
\right.
\end{equation*} 
with  $0\in \Gamma_2(u_j)\cap\mathfrak{g}_{j}^{-1}(0)\cap T _1$,   $\mathrm{h}_ j(0)=|D\mathrm{h}_ j(0)|=|D^2\mathrm{h}_ j(0)|=0 $, 
\[\|u_j\|_{\mathcal{C}^{2,\alpha _0}(\overline{B_{1/2}^+})}\leq \mathrm{C_0}, \quad \displaystyle\sup_j\|\mathrm{h}_j\|_{\mathcal{C}^{2,\widehat\beta-2}(T_1)}\leq \mathrm{C},\] and \begin{align}\label{the8:3}
 |u_j (x_j )|>j|x_j |^{\beta}.\end{align}
For each $j\in \mathbb{R}$, define $$\phi_j(\rho):=\displaystyle\sup_{\rho<r<1/2}r^{2+\alpha_0-\beta}[D^2u_j]_{\alpha_0,B_{r}^{+}\cup T_{r}},$$
then, we have
$$\displaystyle\lim_{\rho\rightarrow0}\phi_j(\rho)=\displaystyle\sup_{0<r<1/2}r^{2+\alpha_0-\beta}[D^2u_j]_{\alpha_0,B_{r}^{+}\cup T_{r}}.$$
Since $0\in \Gamma_2(u_j)\cap\mathfrak{g}_{j}^{-1}(0)\cap T _{1/2}$, it follows that $u_j(0)=|Du_j(0)|=|D^2u_j(0)|=0$, which together with  \eqref{the8:3} leads to
\begin{equation*}\label{the8:4}
    j<\frac{|u_j(x_j)|}{|x_j|^{\beta}}
    \leq\frac{[D^2u_j]_{\alpha_0,B_{|x_j|}^+ \cup T_{|x_j|}}|x_j|^{2+\alpha_0}}{|x_j|^{\beta}}
    \leq \phi_j\left(\frac{|x_j|}{2}\right)
   \leq \displaystyle\lim_{\rho\rightarrow0}\phi_j(\rho).
\end{equation*}
Therefore, 
arguing as in \eqref{the1:5}-\eqref{Eq3.6} in the proof of Theorem \ref{Thm:branching-elliptic}, we deduce that $\rho_j\rightarrow 0$ as $j\rightarrow \infty$.
Define now the blow up sequence $$w_j(x):=\frac{u_j(\rho_jx)}{\rho_j^{\beta}\phi_j(\rho_j)}, \quad \forall \ x \in B_{\frac{1}{2\rho_j}}^+\cup T _{\frac{1}{2\rho_j}}^+.$$
Then, $w_j$  satisfies, in the viscosity sense, the following equation
\begin{equation*}
\left\{
\begin{array}{rcll}
    \widetilde{F}_j(x,D^2w_j)+\langle\widetilde{\mathcal{B}}_j(x),Dw_j\rangle+\widetilde{\varrho_j}(x)|Dw_j|^{\sigma}&=&\widetilde{\mathscr{F}}_j(x,w_j^+,w_j^-) & \text{in} \quad B_{\frac{1}{2\rho_j}}^+,\\
 w_j(x',0)&=& \widetilde{\mathrm{h}}_j(x',0) & \text{on } \quad T _{\frac{1}{2\rho_j}}^+.
\end{array}
\right.
\end{equation*} 
where
$\widetilde{F}_j(x,\mathrm{X}):=\frac{\rho_j^{2-\beta}}{\phi_j(\rho_j)}F\left(\rho_jx,\frac{\phi_j(\rho_j)}{\rho_j^{2-\beta}}\mathrm{X}\right)$, $\widetilde{\mathcal{B}}_j(x):=\rho_j\mathcal{B}_j(\rho_jx)$,
$$ \widetilde{\varrho_j}(x):=\frac{\rho_j^{2-\beta-\sigma+\sigma\beta}}{\phi_j(\rho_j)^{1-\sigma}}\varrho_j{ (\rho_jx)}, \ \  \widetilde{\mathrm{h}}_j(x',0)=\frac{\mathrm{h}_j(\rho _jx',0)}{\rho_j^{\beta}\phi_j(\rho_j)},$$ and
$$\widetilde{\mathscr{F}}_j(x,w_j^+,w_j^-):=\frac{\rho_j^{2-\beta(1- m)}}{\phi_j(\rho_j)^{1-m}}\mathfrak{g}(\rho_{j}x)\left((w_j^+)^m-(w_j^-)^m\right).$$
Observe that $w_j(0)=|Dw_j(0)|=|D^2w_j(0)|=0$, and similarly to  \eqref{the1:6}, we obtain
\begin{align}\label{the8:6}
    [D^2w_j]_{\alpha_0,B_1^+ \cup T_1}
  =  [D^2u_j]_{\alpha_0,B_{\rho_j}^+ \cup T_{\rho_j}}\rho_j^{2+\alpha_0-\beta}\cdot \frac{1}{\phi_j(\rho_j)}
    \geq  \frac{\phi_j(1/j)}{2\phi_j(1/j)}=\frac{1}{2}.
\end{align}

Reasoning analogously to \eqref{the1:7},  we also deduce $|w_j(x)|\leq r^{\beta}$, which implies  $\|w_j\|_{L^{\infty}(B_r^+)}\leq r^{\beta}$ and
 $$|Dw_j(x)|\leq r^{\beta-1},\quad |D^2w_j(x)|\leq r^{\beta-2},  \quad [D^2w_j(x)]_{\alpha,B^+_r \cup T_r}\leq r^{\beta-(2+\alpha_0)}.$$
Hence, since $r>1$, we get $\|w_j\|_{\mathcal{C}^{2,\alpha_0}(B_r^+ \cup T_r)}\leq 4r^{\beta}.$
Thus, by applying the Arzel\'a-Ascoli theorem, there exists some $w_\infty \in \mathcal{C}_{\rm loc}^{2,\alpha_0}(B_r^+ \cup T_r)$ such that,  up to a subsequence,  $w_j\rightarrow w_\infty$ in $\mathcal{C}_{\rm loc}^{2,\alpha_0}(B_r^+ \cup T_r)$. In addition, by the pointwise limit it follows that $w_\infty(x)\leq |x|^{\beta}$,  $$w_\infty(0)=|Dw_\infty(0)|=|D^2w_\infty(0)|=0,$$
and from \eqref{the8:6}, we obtain $[D^2w_\infty]_{\alpha_0,B_1^+ \cup T_1}\geq\frac{1}{2}$.

Now,  letting $j\rightarrow\infty$, likewise in the proof of Theorem \ref{Thm:branching-elliptic}, we conclude that
$$\|\widetilde{\mathcal{B}}_j(x)\|_{L^{\infty}(B_{1}^+;\mathbb{R}^{n})}\rightarrow0, \quad |\widetilde{\varrho_j}(x)|\rightarrow 0, \quad 
|\widetilde{\mathscr{F}}_j(x,w_j^+,w_j^-)|\rightarrow 0,$$
moreover, since $\mathrm{h}_j(0)=|D\mathrm{h}_j(0)|=|D^2\mathrm{h}_j(0)|=0$ and $\mathrm{h}_j\in\mathcal{C}^{2,\widehat\beta-2}(T_1)$, by the definition of $\phi_j$, there exists a constant $\mathrm{C}>0$ such that
$$|\widetilde{\mathrm{h}}_j(x',0)|=\left|\frac{\mathrm{h}_j(\rho _jx',0)}{\rho_j^{\beta}\phi_j(\rho_j)}\right|\leq \frac{\mathrm{C}|\rho _jx'|^{\widehat\beta }}{\rho_j^{\beta}\phi_j(\rho_j)}\leq \frac{\mathrm{C} \rho_j^{\widehat\beta}}{%\phi_j( 
\rho_j^{2 + \alpha_0}}\rightarrow 0,$$
provided that $\widehat \beta > 2 + \alpha_0.$
Since $w_j\rightarrow w_\infty$ locally uniformly in $\mathbb{R}^n_+$, $\|w_j\|_{\mathcal{C}^{2,\alpha_0}(B_r^+ \cup T_r)}\leq 4r^{\beta}$, and $w_j(x',0)=\mathrm{h}_j(x',0)\rightarrow0$, we extend $w_\infty$ continuously  up to the boundary, that is, $w_\infty=0$ {on} $\partial \mathbb{R}^n_+=\{x_n=0\} $.
Again, from \eqref{Eq3.6}, one has
$$\frac{\phi_j(\rho_j)}{\rho_j^{2-\beta}}\leq 2\mathrm{K}_0\rho_j^{\alpha_0}\rightarrow 0 \quad \text{as} \quad j\rightarrow \infty.$$
Thus, the structural conditions $\rm (A_1)-(A_2)$ and $\rm(A_4)$ 
combined with the stability theorem for viscosity
solutions, see \cite[Section 6]{CIL92}, yield that 
\(w_\infty\)  solves, in the viscosity sense the following problem
\begin{equation*}
\left\{
\begin{array}{rcll}
\text{Tr}(\mathfrak{A}(0)D^2w_\infty)&=&0&\text{in } \quad \mathbb{R}^n_+,\\
w_\infty&=&0 &\text{on}\quad \partial \mathbb{R}^n_+.
\end{array}
\right.
\end{equation*}
We now observe that, up to a change of coordinates, $w_\infty$ is harmonic in the half space $\mathbb{R}^n_+$. Then, by the Schwarz Reflection Principle, we can extend $w_\infty$ to the entire space $\mathbb{R}^n$ by the function $\widetilde{w}_\infty$
\begin{equation*}\widetilde{w}_\infty(x',x_n)=
\left\{
\begin{array}{rcll}
w_\infty(x',x_n)&\text{if } \quad x_n\geq 0,\\
-w_\infty(x',-x_n)&\text{if }\quad x_n<0.
\end{array}
\right.
\end{equation*} 
Since \begin{equation*}
\frac{|\widetilde{w}_\infty(x)|}{|x|^{\frac{2+\mu}{1-m}}} \leq \mathrm{C}|x|^{\beta-\frac{2+\mu}{1-m}}\rightarrow 0 \quad \text{as} \quad {|x| \to +\infty},    
\end{equation*}
by the Liouville-type property, we conclude that $\widetilde{w}_\infty$ is  a quadratic polynomial function. Taking into account  that $w_\infty(0)=|Dw_\infty(0)|=|D^2w_\infty(0)|=0$,
we obtain that $w_\infty=0$, which yields a contradiction with \eqref{the8:6}. Therefore, the result holds.

\end{proof}

\section{Further properties of two-phase solutions}\label{Sec6}

We conclude this manuscript by presenting further properties of solutions to \eqref{1.1} in connection with the higher-order singular nodal set.

\subsection{Flipping estimates}    

Next, we prove the flipping estimates by combining a blow-up contradiction argument with the Strong Maximum Principle.

\begin{proof}[{\bf Proof of Theorem \ref{ThmFlipElliptic}}]
Without loss of generality, assume that $x_{0}=0$. Recalling that 
\(
\bar{ \alpha}:=\frac{2+\mu}{1-m}\)
 and setting \(
s_j:=\|u\|_{L^\infty(B_{2^{-j}})},
\)
according to the ideas from \cite{CKS00}, we first show that there exists a constant $\overline{\mathrm C}>0$ such that
\begin{equation}\label{flip-iterative}
s_{j+1}
\leq
\max\Big\{
\overline{\mathrm C}\,2^{-\bar\alpha(j+1)},
\,2^{-\bar\alpha}s_j
\Big\}.
\end{equation}

Suppose by contradiction that \eqref{flip-iterative} fails. Then, for every
$k\in\mathbb N$, there exists $j_k\in\mathbb N$ satisfying
\begin{equation}\label{flip-contradiction}
s_{j_k+1}
>
\max\Big\{
k\,2^{-\bar\alpha(j_k+1)},
\,2^{-\bar\alpha}s_{j_k}
\Big\}.
\end{equation}
Defining the blow-up sequence
\[
u_k(x)
:=
\frac{u(2^{-j_k}x)}
{s_{j_k+1}}
\qquad \text{in } B_1,
\]
since $0\in \Gamma_2(u)$, we have
\(
u_k(0)=0.
\)
Moreover,
\(
\|u_k\|_{L^\infty(B_{1/2})}=1,
\)
and, by \eqref{flip-contradiction},
\[
\|u_k\|_{L^\infty(B_1)}
=
\frac{s_{j_k}}{s_{j_k+1}}
\leq 2^{\bar\alpha}.
\]
A direct computation shows that $u_k$ satisfies, in the viscosity sense, the following equation
\[
F_k(x,D^2u_k)
+
\langle \mathcal{B}_k(x),Du_k\rangle
+
\varrho_k(x)|Du_k|^\sigma
=
f_k(x,u_k)
\qquad \text{in } B_1,
\]
where
\[
F_k(x,\mathrm{X})
:=
\frac{2^{-2j_k}}{s_{j_k+1}}
F\!\left(
2^{-j_k}x,
\frac{s_{j_k+1}}{2^{-2j_k}}\mathrm{X}
\right),
\quad 
\mathcal{B}_k(x)
:=
2^{-j_k}\mathcal{B}(2^{-j_k}x),
\]
\[
\varrho_k(x)
:=
\frac{2^{-j_k(2-\sigma)}}{s_{j_k+1}^{1-\sigma}}
\varrho(2^{-j_k}x),
\mbox{
and
}
f_k(x,u_k)
:=
\mathfrak g(2^{-j_k}x)
\frac{2^{-2j_k}}{s_{j_k+1}^{\,1-m}}
\Big((u_k^+)^m-(u_k^-)^m\Big).
\]
Using \eqref{flip-contradiction}, we obtain
\(
s_{j_k+1}
>
    k\,2^{-\frac{2+\mu}{1-m}(j_k+1)},
\)
which combined with the assumption on $\mathfrak{g}$ yields
\[
|f_k(x,u_k)|\leq \mathrm{C}  \frac{2^{-(2+\mu)j_k}}{s_{j_k+1}^{\,1-m}}
\leq
\frac{\mathrm{C}}{k^{\,1-m}}
\to 0, \ \mbox{ as } k \to \infty.
\]
Hence,
\(
\|f_k\|_{L^\infty(B_1)}
\to 0.
\)
Furthermore,
\(
\|\mathcal{B}_k\|_{L^\infty(B_1)}
\leq
2^{-j_k}\|\mathcal{B}\|_{L^\infty(B_{1})}
\to 0.
\)

Next, using again \eqref{flip-contradiction}, we deduce by $\sigma\in\left(\frac{2m+\mu}{1+\mu+m},1\right)$ that
\[
|\varrho_k(x)|
\leq \frac{2^{\bar{\alpha}(1-\sigma)}2^{-j_{k}(2-\sigma-\bar{\alpha}(1-\sigma))}}{k^{1-\sigma}}\|\varrho\|_{L^{\infty}(B_{1})}\leq 8\|\varrho\|_{L^{\infty}(B_{1})}
\,2^{-j_k[2-\sigma-\bar\alpha(1-\sigma)]},
\]
provided that $\bar{\alpha}\leq 3$. Since $m\in (0,1)$ and the definition of $\bar{\alpha}$ guarantee that
\[
2-\sigma-\bar\alpha(1-\sigma)
=
\frac{(2-\sigma)(1-m)-(2+\mu)(1-\sigma)}{1-m}=\frac{(1+\mu+m)\sigma-(2m+\mu)}{1-m},
\]
the assumption $\sigma>\frac{2m+\mu}{1+\mu+m}$ is equivalent to
\(
(1+\mu+m)\sigma-(2m+\mu)>0,
\)
and
implies
\(
\|\varrho_k\|_{L^\infty(B_1)}
\to 0.
\)

The structural conditions on the coefficients yield
\(
F_k(x, \mathrm{X}) \to \text{Tr}(\mathfrak{A}(0)\mathrm{X})
\)
locally uniformly, due to assumption $(\mathrm{A}_4)$.
By the Arzelà-Ascoli theorem,
up to a subsequence,
\[
u_k\to u_\infty
\qquad\text{locally uniformly in }B_{1}.
\]
Since  \(\mathcal{B}_k\to0\) and \(f_k\to0\)
locally uniformly, the stability theorem for viscosity
solutions, see \cite[Section 6]{CIL92}, implies that
\(u_\infty\)  solves
\begin{equation}\label{limit-equation-elliptic}
\text{Tr}(\mathfrak{A}(0)D^2u_\infty)=0
\qquad \text{in } B_{3/4}.
\end{equation}

Moreover,
\begin{equation}\label{limit-properties}
u_\infty(0)=0,
\qquad
\|u_\infty\|_{L^\infty(B_{1/2})}=1.
\end{equation}

Now, let $k$ be sufficiently large so that $B_{2^{-j_k}}\subset B_{r_0}$. By the hypothesis on the positive phase,
\[
0\leq u_k^{+}(x)
=
\frac{u^{+}(2^{-j_k}x)}
{s_{j_k+1}}
\leq
\frac{\mathrm C_0\,2^{-\bar\alpha j_k}}
{s_{j_k+1}}
\leq
\frac{\mathrm{C}}{k},
\]
where the last inequality follows from \eqref{flip-contradiction}.
Passing to the limit, we obtain that the positive phase becomes asymptotically flat, namely,
\(
u_\infty^{+}\equiv 0.
\)
Hence
\(
u_\infty\leq 0
\ \text{in } B_{3/4}.
\)
Since $u_\infty(0)=0$, the origin is a maximum point of $u_\infty$. By applying the Strong Maximum Principle, see
 \cite[Theorem 3.6]{CC95}, to the equation \eqref{limit-equation-elliptic}, we conclude that $u_\infty$ is constant, and hence,
$u_\infty\equiv 0
\ \ \text{in } B_{3/4}.$
This contradicts the normalization of the blow-up sequence in \eqref{limit-properties}. Therefore,
\eqref{flip-iterative} holds.

Let $j_0\in\mathbb N$ be the smallest integer such that
\(
2^{-j_0}\leq r_0.
\)
Iterating \eqref{flip-iterative}, we obtain
\(
s_j
\leq
\mathrm C\,2^{-\bar\alpha j}
\ \text{for every } j\geq j_0
\)
for some constant $\mathrm C>0$.
Fix $r\in(0,r_0]$ and choose $j\geq j_0$ such that
\(
2^{-(j+1)}
\leq r
\leq 2^{-j}.
\)
Then,
\[
\sup_{B_r}u^{-}
\leq
\|u\|_{L^\infty(B_r)}
+
\mathrm C_0 r^{\bar\alpha}
\leq
s_j+\mathrm C_0 r^{\bar\alpha}.
\]
Using the estimate above,
\[
\sup_{B_r}u^{-}
\leq
\mathrm C\,2^{-\bar\alpha j}
+
\mathrm C_0 r^{\bar\alpha}
\leq
\big(2^{\bar\alpha} \mathrm C+\mathrm C_0\big)r^{\bar\alpha}.
\]
Setting
\(
\mathrm C_1:=2^{\bar\alpha}\mathrm C+\mathrm C_0,
\)
we conclude that
\[
\sup_{B_r}u^{-}
\leq
\mathrm C_1 r^{\bar\alpha}
\qquad \text{for every } r\in(0,r_0].
\]
The proof in the opposite case follows analogously by considering the negative phase and applying the Strong Maximum Principle to the limiting profile. This completes the proof.
\end{proof}

\subsection{Non-degeneracy of solutions}

By constructing a suitable barrier function and applying the Comparison Principle, we finally establish the non-degeneracy of solutions in each phase at higher-order singular nodal points.

\begin{proof}[{\bf Proof of Theorem \ref{Non-degeneracy of solutions}}] 
First, we  prove the estimate involving the supremum. Fix $z \in \{u > 0\}\cap\mathfrak{g}^{-1}(\{0\})$ and choose $r > 0$ such that $B_r(z) \subset B_1$. Without loss of generality, assume that $z=0$. Consider the following auxiliary (barrier) function
\[
\varphi(x) = \kappa |x|^{\bar{\alpha}}, \quad x \in B_r \cap \{u > 0\},
\]
where \(\bar{\alpha} = \frac{2+\mu}{1-m}\). By assumption $\rm (A_2)$, in this case $0\in\mathfrak{g}^{-1}(\{0\})$, then  $0\in F$, and it follows that $\mathfrak{g}(x)\geq \mathrm{c}_{0}|x|^{\mu}$ for every $x\in B_{r}\cap\{u>0\}$). Thus, we can choose \(\kappa > 0\)  so that
\begin{equation}\label{thesisoftheTheorem1.6}
F(x,D^{2}\varphi) + \mathscr{H}(x,D \varphi) - \mathrm{c}_{0} |x|^{\mu} (\varphi^{+}(x))^{m} < 0 
\leq F(x,D^{2} u) + \mathscr{H}(x,D u) - \mathrm{c}_{0} |x|^{\mu} (u^{+}(x))^{m}
\end{equation}
in \(B_r \cap \{u>0\}\). For this, we observe that
\begin{itemize}
\item[] \(D\varphi(x)=\kappa\bar{\alpha}|x|^{\bar{\alpha}-2}x\),
\quad \(D^{2}\varphi(x)=\kappa\bar{\alpha}|x|^{\bar{\alpha}-2}\left(\mathrm{Id}_{n}+(\bar{\alpha}-2)\frac{x}{|x|}\otimes \frac{x}{|x|}\right)\),
\item[] \qquad \qquad \qquad and  \(\mathcal{M}_{\lambda,\Lambda}^{+}(D^2\varphi(x))=\Lambda\kappa\bar{\alpha}(n+\bar{\alpha}-2)|x|^{\bar{\alpha}-2}\).
\end{itemize}
Consequently, we have the following estimate for the left-hand side of \eqref{thesisoftheTheorem1.6}
\begin{align*} 
F(x, D^{2}\varphi)&+ \mathscr{H}(x,D \varphi) - \mathrm{c}_{0} |x|^{\mu} \varphi_{+}^{m}(x)
\leq\mathcal{M}_{\lambda,\Lambda}^{+}(D^{2}\varphi)+\|\mathcal{B}\|_{L^{\infty}(B_1;\mathbb{R}^{n})}|D\varphi|\\
&+\|\varrho\|_{L^{\infty}(B_1)}|D\varphi|^{\sigma}-\mathrm{c}_{0}|x|^\mu\varphi_{+}^{m}(x)\\
&\leq\kappa\bar{\alpha}\Lambda(n+(\bar{\alpha}-2))|x|^{\bar{\alpha}-2}+\kappa\bar{\alpha}\|\mathcal{B}\|_{L^{\infty}(B_1;\mathbb{R}^{n})}|x|^{\bar{\alpha}-1}\\ 
&+\|\varrho\|_{L^{\infty}(B_1)}(\kappa\bar{\alpha})^{\sigma}|x|^{\sigma(\bar{\alpha}-1)}-\mathrm{c}_{0}\kappa^{m}|x|^{\mu +\bar{\alpha}m}\\ 
&\leq \Big[\kappa\bar{\alpha}\left[\Lambda(n+\bar{\alpha}-2) + \|\mathcal{B}\|_{L^{\infty}(B_1;\mathbb{R}^{n})}+\|\varrho\|_{L^{\infty}(B_1)}\bar{\alpha}^{\sigma-1}\right]-\mathrm{c}_{0}\kappa^{m}\Big]|x|^{\mu+m\bar{\alpha}}, 
\end{align*}
This expression is negative provided that
\[
\kappa\bar{\alpha}\Big(\Lambda(n+\bar{\alpha}-2) 
+ \|\mathcal{B}\|_{L^{\infty}(B_1;\mathbb{R}^{n})}
+ \|\varrho\|_{L^{\infty}(B_1)}\bar{\alpha}^{\sigma-1}\Big)
<\mathrm{c}_{0}\kappa^{m}.
\]
Therefore, it is sufficient to choose  \(\kappa\) such that
\begin{equation*}
\kappa^{1-m}
<\frac{\mathrm{c}_{0}}{\bar{\alpha}\left( \Lambda(n +(\bar{\alpha} - 2) 
+ \|\mathcal{B}\|_{L^{\infty}(B_1; \mathbb{R}^{n})} 
+ \bar{\alpha}^{\sigma - 1} \|\varrho\|_{L^{\infty}(B_1)} \right)}.
\end{equation*}
Equivalently, choosing \(\kappa>0\) such that
\begin{equation}\label{kappachoice}
\kappa < \min\left\{ 1, \frac{\mathrm{c}_{0}^{\frac{1}{1-m}}}{\bar{\alpha}^{\frac{1}{1-m}} 
\left[ \Lambda(n + \bar{\alpha} - 2) 
+ \|\mathcal{B}\|_{L^{\infty}(B_1; \mathbb{R}^{n})} 
+ \bar{\alpha}^{\sigma - 1} \|\varrho\|_{L^{\infty}(B_1)} \right]^{\frac{1}{1-m}}} \right\},
\end{equation}
we obtain
\[
F(x,D^{2}\varphi) + \mathscr{H}(x,D \varphi) - \mathrm{c}_{0} |x|^{\mu} \varphi_{+}^{m}(x)<0.
\]

 Under these conditions, we claim that there exists \(y \in \partial(B_r \cap \{u>0\})\) such that \(u(y) > \varphi(y)\). Indeed, if
\[
\partial(B_r \cap \{u>0\})\cap\{ u(x) > \varphi(x)\}=\emptyset,
\]
then,  \eqref{thesisoftheTheorem1.6} allows us to conclude, by the Comparison Principle \ref{ComPri} for $\mathfrak{c}(x)=-\mathrm{c}_{0}|x|^{\mu}$ and $\mathcal{F}(t)=(t^{+})^{m}$, that \(u\leq \varphi\) in \(B_r \cap \{u>0\}\). However, \(u(0) > 0 = \varphi(0)\), which is a contradiction. 
Therefore, there exists
\[
y \in \partial(B_r \cap \{u>0\})\cap \{u>\varphi\},
\]
that is, \(u(y) > \varphi(y)\). As \(u \leq \varphi\) on \(\partial\{u>0\} \cap B_r\), we must have \(y \in \partial B_r \cap \{u>0\}\). Therefore,
\[
\sup_{\partial B_r \cap \{u>0\}} u(x) \,\geq\,u(y) \,>\, \varphi(y) \,=\, \kappa r^{\bar{\alpha}}.
\]
For $x_{0}\in \Gamma_2(u)\cap \mathfrak{g}^{-1}(\{0\})$, take a sequence \(z_{j} \to x_0\) in \(\{u>0\} \cap B_{1/2}(x_0)\cap \mathfrak{g}^{-1}(\{0\})\) and pass to the limit by continuity.

Finally, we observe that the estimate involving the infimum follows by the same argument applied in the negative phase. More precisely, let $z\in\{u<0\}\cap \mathfrak{g}^{-1}(\{0\})$ and choose $r>0$ such that $B_r(z)\subset\{u<0\}$. Without loss of generality, we may assume that $z=0$, and consider the barrier function
\[
\psi(x)=-\kappa|x|^{\bar{\alpha}}, \qquad x\in B_r(z)\cap\{u<0\},
\]
for the same constant $\kappa>0$ such that \eqref{kappachoice} holds.
Proceeding exactly as above, it is possible to check that $\psi$ solves
\begin{equation*}
F(x,D^{2}\psi) + \mathscr{H}(x,D \psi) - \mathrm{c}_{0} |x|^{\mu}(- (\psi^{-}(x))^{m})>0  \quad \text{in}\quad B_{r}\cap \{u<0\}
\end{equation*}
and by the condition ${\rm (A_{2})}$, $u$ satisfies
\begin{equation*}
F(x,D^{2}u) + \mathscr{H}(x,D u) -\mathrm{c}_{0} |x|^{\mu}(- (u^{-}(x))^{m})\leq 0 \quad \text{in}\quad B_{r}\cap \{u<0\}.
\end{equation*}
Again, we may apply the Comparison Principle \ref{ComPri}  with $\mathfrak{c}(x)=-\mathrm{c}_{0}|x|^{\mu}$ and $\mathcal{F}(t)=-(t^{-})^{m}$ to conclude that there exists a point $y\in \partial(B_{r}\cap\{u<0\})$ such that $u(y)<\psi(y)$. Since $u= 0\geq \psi$ on $\partial\{u<0\}\cap B_{r}$, it follows that $y\in \partial B_{r}\cap\{u<0\}$. Therefore,
\[
\inf_{\partial B_{r}\cap\{u<0\}}u
\leq u(y)
<\psi(y)
=-\kappa r^{\bar{\alpha}}.
\]
For $x_{0}\in \Gamma_2(u)\cap \mathfrak{g}^{-1}(\{0\})$, let $\{z_{j}\}\subset\{u<0\}\cap B_{1/2}(x_{0})\cap \mathfrak{g}^{-1}(\{0\})$ be a sequence such that $z_{j}\to x_0$. Passing to the limit by continuity, we conclude that
\[
\inf_{B_{r}(x_0)\cap\{u<0\}}u(x)\leq -\kappa r^{\bar{\alpha}},
\]
which completes the proof.

\end{proof}

\subsection*{Data Availability Statement} \small Data sharing is not applicable to this article as no datasets were generated or analyzed
	during the current study.

\subsection*{Conflict of Interest Declaration}\small The authors declare that they have no conflict of interest.

\subsection*{AI disclosure} \small
In preparing this paper, the authors used the free version of ChatGPT solely for reviewing, improving, and editing the manuscript. ChatGPT was not used in developing the structure of the paper, including the Introduction, or in formulating, developing, or validating the mathematical arguments and proofs. All mathematical content, arguments, proofs, and conclusions were conceived, developed, and independently verified by the authors.

\subsection*{Acknowledgments}

J. da Silva Bessa has been supported by FAPESP-Brazil under Grant No. 2023/18447-3.
J.V. da Silva has received partial support from CNPq-Brazil under Grant No. 307131/2022-0,  Chamada CNPq/MCTI Nº 10/2023 - Faixa B - Grupos Consolidados under Grant 420014/2023-3, and by  FAPESP-Brazil under the Grant No.  2025/09344-1-Special Programs-Special Projects-First Projects-Call for Proposals (2025)-1st Cycle. 
M. Soares has been supported by CNPq-Brazil under the Grants No. 303.154/2025-0 and  No. 152.786/2025-2. Y. Hu is deeply grateful to FAPESP for funding her doctoral scholarship during the conduct of this research (Grant 2026/03110-1).

{\small{

}}
\bigskip

	\noindent\textsc{Junior da Silva Bessa}\\
	Departamento de Química\\
	Faculdade de Educação de Itapipoca - FACEDI\\
	Universidade Estadual do Ceará - UECE\\
	Rua da Universidade, s/n°, CEP 62505-090, Itapipoca, CE, Brazil\\
	\noindent \texttt{silva.bessa@uece.br}
	\bigskip
    
	\noindent\textsc{Jo\~{a}o Vitor da Silva and Yuwei Hu}\\
	Departamento de Matem\'{a}tica\\
	Instituto de Matem\'{a}tica, Estat\'{i}stica e Computa\c{c}\~{a}o Cient\'{i}fica - IMECC\\
	Universidade Estadual de Campinas - Unicamp\\
	Rua S\'{e}rgio Buarque de Holanda, 651, CEP 13083-859, Campinas, SP, Brazil\\
	\noindent \texttt{jdasilva@unicamp.br} and \texttt{y290757@dac.unicamp.br}

	\bigskip
	
	\noindent\textsc{Mayra Soares}\\
	Departamento de Matem\'{a}tica,\\ Universidade de Bras\'{i}lia - UnB\\
	Instituto Central de Ci\^{e}ncias, Campus Darcy Ribeiro,\\
	70910-900, Asa Norte, Bras\'{i}lia, Distrito Federal, Brasil\\
	\noindent\texttt{mayra.soares@unb.br}
	
\end{document}